\documentclass[11pt,letterpaper]{amsart}

\usepackage[utf8]{inputenc}
\usepackage[english]{babel}
\usepackage{amscd,amssymb,amsthm,amsmath,csquotes}
\usepackage{marginnote} 
\usepackage{mathrsfs}
\usepackage{quiver}
\usepackage{float}
\usepackage[margin=1.1in]{geometry}
\usepackage{comment}
\usepackage{mathtools}
\usepackage{upgreek}
\usepackage{faktor}
\usepackage{changepage}
\usepackage{rotating}
\usepackage{enumerate}
\usepackage[all,cmtip]{xy}

\usepackage[
backend=biber,
style=alphabetic,
sorting=nyt,
maxbibnames=99, minbibnames=99
]{biblatex}

\newtheorem{theorem}{Theorem}[section]
\newtheorem{lemma}[theorem]{Lemma}

\newtheorem{proposition}[theorem]{Proposition}
\newtheorem{corollary}[theorem]{Corollary}

\theoremstyle{definition}
\newtheorem{definition}[theorem]{Definition}
\newtheorem{example}{Example}[section]
\theoremstyle{remark}
\newtheorem*{remark}{Remark}

\makeatletter
\newcommand{\colim@}[2]{%
  \vtop{\m@th\ialign{##\cr
    \hfil$#1\operator@font colim$\hfil\cr
    \noalign{\nointerlineskip\kern1.5\ex@}#2\cr
    \noalign{\nointerlineskip\kern-\ex@}\cr}}%
}
\newcommand{\colim}{%
  \mathop{\mathpalette\colim@{\rightarrowfill@\textstyle}}\nmlimits@
}
\makeatother

\makeatletter
\newcommand{\hocolim@}[2]{%
  \vtop{\m@th\ialign{##\cr
    \hfil$#1\operator@font hocolim$\hfil\cr
    \noalign{\nointerlineskip\kern1.5\ex@}#2\cr
    \noalign{\nointerlineskip\kern-\ex@}\cr}}%
}
\newcommand{\hocolim}{%
  \mathop{\mathpalette\hocolim@{\rightarrowfill@\textstyle}}\nmlimits@
}
\makeatother

\usepackage{hyperref}
\hypersetup{
    colorlinks,
    citecolor=black,
    filecolor=black,
    linkcolor=black,
    urlcolor=black
}
\newcommand{\Tot}{\operatorname{Tot}}
\newcommand{\Hom}{\operatorname{Hom}}

\newcommand{\id}{\operatorname{id}}

\newcommand{\Ind}{\operatorname{Ind}}
\newcommand{\Res}{\operatorname{Res}}
\newcommand{\Cn}{\operatorname{Cone}}
\newcommand{\A}{\mathcal{A}}
\newcommand{\B}{\mathcal{B}}
\newcommand{\C}{\mathcal{C}}

\newcommand{\F}{\mathcal{F}}

\newcommand{\cS}{\mathcal{S}}
\newcommand{\T}{\mathcal{T}}

\newcommand{\op}[1]{{#1}^{op}}
\newcommand{\Mod}{\operatorname{-Mod}}
\newcommand{\RMod}{\operatorname{Mod-}}

\newcommand{\hot}{\operatorname{Hot}}

\newcommand{\Ker}{\operatorname{Ker}}
\newcommand{\ev}{\operatorname{ev}}
\newcommand{\Coker}{\operatorname{Coker}}
\newcommand{\img}{\operatorname{Im}}

\newcommand{\Img}{\operatorname{Im}}

\newcommand{\tow}[1]{\overset{#1}{\to}}
\newcommand{\Hm}[1]{\Hom_{#1}}
\newcommand{\tor}[1]{{#1}^{\operatorname{tor}}}

\newcommand{\kn}{k[t]/(t^{n+1})}
\newcommand{\dsi}{D^{\operatorname{si}}}
\newcommand{\Gr}{\operatorname{Gr}}
\title{Filtered derived categories of curved deformations}

\author{Alessandro Lehmann}
\address[Alessandro Lehmann]{Universiteit Antwerpen, Departement Wiskunde, Middelheimcampus,
Middelheimlaan 1,
2020 Antwerp, Belgium}
\email{alessandro.lehmann@uantwerpen.be}

\address{SISSA, Via Bonomea 265, 34136 Trieste TS, Italy}
\email{alehmann@sissa.it}

\author{Wendy Lowen} 
\address[Wendy Lowen]{Universiteit Antwerpen, Departement Wiskunde, Middelheimcampus,
Middelheimlaan 1,
2020 Antwerp, Belgium}
\email{wendy.lowen@uantwerpen.be}

\thanks{
This project has received funding from the European Research Council (ERC) under the European Union’s Horizon 2020 research and innovation programme (grant agreement No. 817762).
}

\makeatletter
\@namedef{subjclassname@2020}{
  \textup{2020} Mathematics Subject Classification}
\makeatother


\makeindex 
\begin{document}

\begin{abstract}
We propose a solution to the ``curvature problem'' from \cite{MoritaDef}, \cite{keller2009nonvanishing}, \cite{LowenCurvature} for infinitesimal deformations. 
Let $k$ be a field, $A$ a dg algebra over $k$ and $A_n = A[t]/(t^{n+1})$ a cdg algebra over $R_n =  k[t]/(t^{n+1})$, $n \geq 0$, with reduction $A_n/tA_n = A$. We define the \emph{$n$-derived category} $D^n(A_n)$ as the quotient of the homotopy category by the modules for which all quotients appearing in the associated graded object are acyclic. We prove this to be a compactly generated triangulated category with a semiorthogonal decomposition by $n +1$ copies of $D(A)$, in which Positselski's semiderived category embeds admissibly.
\end{abstract}

\maketitle

\tableofcontents

\section{Introduction}\label{parintro}

Dg categories have become increasingly important in contemporary algebra and geometry, and in particular in algebraic geometry, representation theory and neighboring fields \cite{kellerICM}. They typically model (derived) categories of sheaves or modules, and feature prominently in both derived algebraic geometry and noncommutative algebraic geometry \cite{DAGTo, toen_derived_morita, bonvdbgenerators, HodgeMirror, kaledinhodge}. From this perspective, it is surprising that no comprehensive infinitesimal deformation theory of dg categories exists today, in contrast to the situation for other categorical models of spaces, like abelian categories \cite{AbDef} and prestacks \cite{van2023box}. In the present paper, we lay the foundations for such a new theory, focusing on the compactly generated case, that is, the case of categories generated by a dg algebra. 

It has been known for over 15 years that in general, deformation theory of dg algebras faces the \emph{curvature problem}, stemming from the following observations:
\begin{enumerate}
    \item According to the Hochschild complex, deformations of a dg algebra include \emph{curved} dg algebras and in general it is not possible to realise the full cohomology by means of dg deformations, even up to Morita equivalence \cite{MoritaDef, DAGX}; 
    \item In general, such curved dg algebras have no conventional derived category, and alternative derived categories may vanish for deformations \cite{Positselski_2011, keller2009nonvanishing}.
\end{enumerate}

As a consequence of these two points taken together, already the classical and fundamental algebro-geometric concept of a first order deformation seems to have no good counterpart in the realm of dg categories. This is particularly problematic from the point of view of noncommutative algebraic geometry, where the key philosophy is that categories make up for the loss of classical spaces. However, (2) seems to tell us that not only do we have to deal with a lack of points, in the presence of curvature we also have to deal with a lack of sheaves.

Before presenting our proposed solution to this problem, we now elaborate on these two points, starting with (2).

Curved dg algebras (cdg algebras) are a generalization of differential graded algebras introduced  originally in \cite{poquadual} in the context of quadratic duality. Since these objects have differentials that a priori do not square to zero, homological algebra - and in particular, derived categories - over them cannot be developed in the standard way. Alternative derived categories ``of the second kind" over curved dg algebras have been introduced and studied by Positselski \cite{Positselski_2011}. Whereas these capture for instance the geometric setup of matrix factorizations very well, it has been observed that no general notion can exist which actually \emph{extends} the classical derived category of a dg algebra in a natural way. Further, derived categories of the second kind have a tendency to trivialize \cite{keller2009nonvanishing}.
This is also true of the semiderived categories from \cite{Positselski_2018}, which interpolate between derived categories of the first and second kind in filtered settings like the deformation setup.

On the other hand, when dealing with deformations of honest dg algebras, one may expect the situation to be better because the curvature is ``small"; this brings us to point (1) above. In the deformation setup, it is indeed sometimes possible to get rid of curvature by choosing generators and Hochschild representatives in a particular way. An example is given by deformations of nice schemes, which are shown to be compactly generated in \cite{LOWEN2013441}.

However, when one turns to arbitrary dg algebras, no general solution along these lines exists \cite{MoritaDef}. Precisely, it was shown in loc. cit.  that in general there exist Hochschild classes which cannot be realized by means of Morita deformations.

The goal of the present paper is to solve the curvature problem by taking an altogether different and novel approach to the definition of derived categories of curved deformations of dg algebras. Rather than trying to model the behavior of classical derived categories under, say, algebra deformation or abelian deformation, we look for a new type of derived category associated to a curved deformation which naturally fits into a ``categorified square zero extension'' (in the first order case) by its very construction. 

The suitable notion to realize this turns out to be surprisingly simple and in principle available as soon as the curvature is small with respect to a suitable filtration, making the quotients appearing in the associated graded object uncurved. Indeed, in this case one can consider filtered quasi-isomorphisms, which by definition induce ordinary quasi-isomorphisms on the associated graded. In an uncurved context the use of filtered quasi-isomorphisms goes back at least to \cite{IllusieCot}. For cdg algebras it has been used in \cite{belliermillès2020homotopy} and \cite{calaque2021lie} and for coalgebras in \cite{Positselski_2011}. However, to the best of our knowledge, filtered derived categories of curved dg algebras have not yet been proposed or studied in the literature before.

In this paper, we put ourselves in the classical setup of infinitesimal deformations of order $n$, which is already of interest. In the rest of this introduction, we take $k$ to be a field, but most of our results hold over a commutative ground ring. Precisely, we let $A$ be a dg algebra over $k$ and $A_n = A[t]/(t^{n+1})$ a curved dg algebra over $R_n =  \kn$, with reduction $A_n/tA_n = A$. In \S \ref{parn} we define the \emph{$n$-derived category} $D^n(A_n)$ as the quotient of the homotopy category by the modules for which all quotients appearing in the associated graded object are acyclic (Definition \ref{deffiltered}). 

The first main application of $n$-derived categories we want to realize is the following square of isomorphisms for a dg algebra $A$:
\begin{equation*}\label{eqsquare}
    \xymatrix{{\mathsf{HH}^2(A)} \ar[d]_{\alpha} \ar[r]^-{\mu} & {\mathrm{cDef}^{1-\mathrm{Mor}}_A(k[\epsilon])} \ar[d]^{\beta} \\
    {\mathsf{HH}^2(D_{\mathrm{dg}}(A))} \ar[r]_-{\eta} & {\mathrm{Def}^{\mathrm{sq}-0}_{D_{\mathrm{dg}}(A)}(k[\epsilon])}
    }
\end{equation*}
Here, the isomorphism $\alpha$ is the characteristic morphism for Hochschild cohomology, see \cite{lowchar}, and ${\mathrm{cDef}^{1-\mathrm{Mor}}_A(k[\epsilon])}$ stands for curved dg deformations of $A$ up to $1$-Morita equivalence, that is, equivalence of corresponding $1$-derived categories. The fact that there is indeed a well-defined isomorphism $\nu$ is shown in the companion paper \cite{cMoritaDef}, combining the tools from \cite{MoritaDef} with the ingredients from the present paper. The arrow $\eta$ will be formally introduced in \cite{FutDef} for more general pretriangulated dg categories than the dg derived category $D_{\mathrm{dg}}(A)$, and realizes Hochschild cocycles through categorified square zero extension; this is work in progress. The arrow $\beta$ realises the $1$-derived category of a curved deformation as such a square zero extension, see the depiction \eqref{eqrecol} below.
In future work, we will further relate this to the approach put forth in \cite{KaledinLowen, BKL} to abelian deformations in the absolute setup using Mac Lane rather than Hochschild cohomology, as well as to the approach from \cite{GENOVESEtstr, genovese2022tstructures, Tst3} to deformations of pretriangulated dg categories with t-structures and enough derived injectives. Moreover, the relation with categories of twisted objects (\cite{FilteredAinf, defotriamodels, vandekreeke2023}) remains to be fully elucidated.

Let us explain some of the technical motivations behind our definition of the $n$-derived category. In \cite{keller2009nonvanishing}, the vanishing of several ``derived'' categories of cdg algebras was shown assuming only some very basic axioms, key among which was the fact that any short exact sequence should give rise to a triangle in the quotient. This is the property that we negate in order to obtain a derived category that is guaranteed not to vanish. Indeed, in our case a short exact sequence of $A_n$-modules 
$0 \to M \to N \to L \to 0$ is a triangle precisely when for all $i=1, \ldots, n$ the induced sequence \[
0 \to \frac{M}{t^iM}\to \frac{N}{t^iN} \to \frac{L}{t^iL} \to 0
\] is exact. Vice versa, all triangles up to isomorphism are of this form.

Our first main result is the following (see Theorem \ref{propcompgen} and Theorem \ref{filtrationscat} in the text):

\begin{theorem}\label{thm1}
    The category $D^n(A_n)$ is compactly generated and allows for a semiorthogonal decomposition into $n+1$ copies of $D(A)$.  
\end{theorem}
The generators we describe are closely related to the ``uncurving'' generators from \cite{LowenCurvature}, which however in the context of Morita deformations from loc. cit.  require an additional hypothesis in order to exist. 

The semiorthogonal decomposition in Theorem \ref{thm1} stems from a recollement (see the remark after Proposition \ref{factors})
\begin{equation} \label{eqrecol}
\begin{tikzcd}
	{D^{n-1}(A_{n-1})} & {D^{n}(A_{n})} & {D(A)}
	\arrow["\iota_{n}", hook, from=1-1, to=1-2]
	\arrow["{\Img t^n}"{pos=0.6}, from=1-2, to=1-3]
	\arrow[curve={height=-28pt}, from=1-2, to=1-1]
	\arrow[curve={height=28pt}, from=1-2, to=1-1]
	\arrow[curve={height=28pt}, from=1-3, to=1-2]
	\arrow[curve={height=-28pt}, from=1-3, to=1-2]
\end{tikzcd}
\end{equation}
in which the essential image of $\iota_n$ is given by the $A_n$-modules $M$ for which $t^nM$ is acyclic, reminiscent of the setup of \emph{abelian} deformations from \cite{AbDef}. In the fist order case, we thus obtain the categorified square zero extension to be used in the square \eqref{eqsquare}.

Theorem \ref{thm1} suggests a resemblance between the $n$-derived category and the construction of a categorical resolution of singularities of \cite{Kuznetsov_2014}, since they both come equipped with a semiorthogonal decomposition into the same factors.
Precisely, we show:

\begin{proposition}
    Let $A$ be a smooth dg algebra and suppose $A_n$ is a dg deformation of $A$ ($c = 0$). Then the $n$-derived category $D^n(A_n)$ is a categorical resolution of singularities of the ordinary derived category $D(A_n)$.
\end{proposition}

This result offers a novel perspective on our construction: indeed for a curved deformation $A_n$, while there is no good notion of a derived category, there exists a perfectly well-defined ``resolution'' of the nonexistent derived category of $A_n$.
Note that the relevance of filtered modules to categorical resolutions is already clear from \cite{KaKu}, and a construction similar to ours was recently carried out in \cite{dedeyn2023categorical}. However, establishing the precise relation in the geometric case requires a more careful analysis of the gluing bimodules associated to our decomposition.

The category of $A_n$-modules allows for an explicit construction of semifree (and a fortiori $n$-homotopy projective) resolutions (see \S \ref{parres}), and in fact has a ``projective'' model structure (see \S \ref{parmodel}). These results are obtained along the lines of \cite{Derivingdgcat} and \cite{relative_model} respectively. It also has homotopy injective resolutions, the construction of which is detailed in the appendix \ref{secinj}.

It is well known that in the case of formal deformations, one should look at either torsion or complete derived categories rather than ordinary ones \cite{TARRIO19971,Dwyer2002CompleteMA, HomologyComp, POSITSELSKIMGM, LowenCurvature}. In \S \ref{partorsion}, we extend most of our results to the formal setting of a curved $k[[t]]$-deformation $A_t$ of $A$ by introducing the \emph{$t$-derived category of torsion modules} $\tor{D^t(A_t)}$ (Definition \ref{deftderived}), which constitutes an $\infty$-categorical colimit of the $n$-derived categories.

Our final main goal in the present paper is to compare our filtered derived categories with Positselski's semiderived category from \cite{Positselski_2018}. Recall that for an uncurved deformation $A_n$,
the semiderived category $\dsi(A_n)$ allows for an embedding
$D(A_n) \rightarrow \dsi(A_n)$ of the ordinary derived category. In the formal case, one actually has an equivalence of categories $\tor{D(A_t)} \cong \dsi(A_t^{\operatorname{co}})$ so\footnote{We denote by $\dsi(A_t^{\operatorname{co}})$ the semiderived category of $A_n$ with comodule coefficients, in the terminology of \cite{Positselski_2018}.} in this case, the semiderived category does generalize the kind of derived category one traditionally considers in the uncurved formal setup. It is highly likely that $\dsi(A_t^{\operatorname{co}})$ is further equivalent to the torsion derived category obtained in \cite{LowenCurvature} after uncurving, which is always possible in the formal case for a dga which is cohomologically bounded above (see also \cite{Blanc_Katzarkov_Pandit_2018} for a treatment in the language of formal moduli problems introduced in  \cite{DAGX}).

In the present paper, we show (Corollaries \ref{corsemider} and \ref{semiform}):
\begin{theorem}
There are admissible embeddings $\dsi(A_n) \rightarrow D^n(A_n)$ in the infinitesimal setup, and a  left admissible embedding $\dsi(A_t^{\operatorname{co}}) \rightarrow \tor{D^t(A_t)}$ in the formal setup.
\end{theorem}

Finally, it is a natural question to ask how the definition of the $n$-derived category generalizes to deformations over local artinian rings which have dimension higher than $1$, or to objects with curvature which is small with respect to some filtration - for example Fukaya type categories \cite{FOOO1} or algebras associated to non-flat vector bundles \cite{BANDIERA2023109358}. The obvious proposal is to simply quotient out from the homotopy category the modules which are filtered acyclic; in the local artinian setting, this corresponds to considering the $\mathfrak{m}$-adic filtration, where $\mathfrak{m}$ is the unique maximal ideal. It turns out however that for most technical results it is very important that the ideal defining the filtration is principal, and without this assumption many arguments break down. It is therefore likely that a more refined approach is required for the general case.

\vspace{8pt}

\noindent \emph{Acknowledgements.}
The first named author thanks Nicolò Sibilla for the many interesting conversations. He is also grateful to Leonid Positselski for the encouragement and to Bernhard Keller, Matt Booth, Emanuele Pavia, Timothy De Deyn, Barbara Fantechi, Dmitry Kaledin, Alexander Kuznetsov and Marco Manetti for the insightful discussions. 
\section{Preliminaries}\label{parprelim}
\subsection{Conventions}
All graded objects are assumed to be graded by the integers; we employ cohomological grading, i.e. differentials increase the degree. All rings and algebras are associative and unital. If $\C$ is a dg category, we will denote by $Z^0\C$ it's underlying category, and with $H^0\C$ its homotopy category; the complex of morphisms in the category $\C$ will be denoted $\Hm{\C}(X,Y)$. If $\T$ is a triangulated category linear over a ring $R$, with $\Hm{\T}(X,Y)$ we will denote the $R$-module of morphisms in $\T$.
\subsection{Triangulated and derived categories}
We begin by recalling some basic facts about derived and triangulated categories, as well as fixing some notation; for a comprehensive introduction see  \cite{krause2009localization} or the first section of \cite{Positselski_2011}.
\subsubsection{Localizations of triangulated categories}
 Let $\T$ be a triangulated category admitting small coproducts and $\cS\subseteq \T$ a thick subcategory, i.e. a triangulated subcategory closed under retracts. The subcategory $\prescript{\perp}{}{\cS}$ is defined as the full subcategory of all objects $X$ such that $\Hm{\T}(X,Y)=0$ for all $Y\in \cS$; analogously, $\cS^\perp \subseteq \T$ is given by the objects $Y$ for which $\Hm{\T}(X,Y)=0$ for all $X\in \cS$. We will denote by $\T/\cS$ the Verdier quotient of $\T$ by $\cS$. The following result is classical, see for example \cite[Lemma 4.8.1, Proposition 4.9.1]{krause2009localization} for a proof.
\begin{proposition}\label{krause}
    Let $X, Y \in \T$. If either $X\in \prescript{\perp}{}{\cS}$ or $Y\in \cS^\perp$, then the natural map \[\Hm{\T}(X,Y) \to \Hm{\T/\cS}(X,Y)\]  is an isomorphism. Moreover, the following are equivalent:
    \begin{itemize}
        \item The quotient functor $\T\to \T/\cS$ admits a left adjoint, which is automatically fully faithful;
        \item for any $X\in \T$, there exists a triangle in $\T$ \[P\to X \to S \to P[1]\] with $P\in  \prescript{\perp}{}{\cS}$ and $S\in \cS$.
        \item The composition \[
        \prescript{\perp}{}{\cS} \hookrightarrow \T \to \T/\cS
        \] is an equivalence. 
    \end{itemize}
\end{proposition}
The left adjoint to the quotient functor is constructed by associating to the object $X$ the object $P$ as per the triangle above. As a consequence of this construction, the essential image of the left adjoint is the subcategory $\prescript{\perp}{}{\cS}\subseteq \T$.  The dual statement about the existence of a right adjoint also holds with $\cS^\perp$ in place of $\prescript{\perp}{}{\cS}$ and the arrows in the triangle reversed. 

\subsubsection{Semiorthogonal decompositions}A \emph{semiorthogonal decomposition}  \[
\T=\langle \cS_0, \cS_1\rangle 
\] of a triangulated category $\T$ is by definition given by a pair of triangulated subcategories $\cS_0, \cS_1\subseteq \T$ such that $\Hm{\T}(S_1, S_0)=0$ for all $S_0\in \cS_0$ and $S_1\in \cS_1$ and such that for every $X\in \T$ there exists a triangle \[
S_1 \to X \to S_0 \to S_1[1]
\] with $S_0\in \cS_0$ and $S_1\in \cS_1$; the definition generalizes straightforwardly to decompositions having more that one factor. A triangulated subcategory $\cS\subseteq \T$ is said to be \emph{left admissible} if the inclusion admits a left adjoint, \emph{right admissible} if the inclusion admits a right adjoint and \emph{admissible} if it is both left and right admissible; any right admissible subcategory $\cS$ defines a semiorthogonal decomposition \[
\T=\langle \cS^\perp, \cS \rangle
\] and dually for left admissible subcategories.
\subsubsection{Compactly generated triangulated categories}
If $\T$ is a triangulated category admitting small coproducts, an object $C\in \T$ is said to be compact  if $\T(C,-)$ commutes with small coproducts. The category $\T$ is said to be \emph{compactly generated} if there exists a set $\C\subseteq \T$ of compact objects generating $\T$, i.e. such that for all $X\in \T$ one has \[
    X=0 \Leftrightarrow \Hm{\T}(C[n], S)=0 \text{ for all } n\in \mathbb{Z} \text{ and all } C \in \C. 
    \]
A triangulated subcategory is said to be localizing if it is closed under small coproducts. It can be proven (\cite[Lemma 2.2.1]{schwedestablemodel}) that, provided that $\C\subseteq\T$ is composed of compact objects, the condition that $\C$ generates $\T$ is equivalent to the fact that $\T$ coincides with the minimal localizing subcategory containing $\C$.
\subsubsection{Derived categories of dg algebras}
Let $A$ be a dg algebra. We will denote by $A\Mod$ the dg category of left $A$-modules and with $\hot(A)$ the homotopy category $H^0A\Mod$. The category $\hot(A)$ has the structure of a triangulated category, with triangles given by graded split short exact sequences i.e. short exact sequence that split as sequences of graded $A$-modules. The derived category $D(A)$ is the quotient of $\hot(A)$ by the subcategory $\operatorname{Ac}\subseteq \hot(A)$ given by the acyclic $A$-modules. An $A$-module $M$ is said to be \emph{homotopy projective} if it lies in $\prescript{\perp}{}{\operatorname{Ac}}$ and \emph{homotopy injective} if it lies in $\operatorname{Ac}^\perp$; the quotient $\hot(A)\to D(A)$ admits both a left adjoint $\mathbf{p}$ and a right adjoint $\mathbf{i}$ which are fully faithful and whose essential images are given respectively by the homotopy projective and homotopy injective modules.
The derived category is always compactly generated; we include the elementary proof of this fact since it will serve as a blueprint for the proof in the curved case.
\begin{proposition}
    The free module $A\in A\Mod$ is a homotopy projective compact generator of $D(A)$.
\end{proposition}
\begin{proof} 
    Recall the isomorphism \[
    \Hm{A\Mod}(A,M)\cong M
    \] for any $A$-module $M$; if $M$ is acyclic it's clear that \[\Hm{\hot(A)}(A[n],M)\cong H^{-n}M=0\] so $A$ is homotopy projective and $\Hm{\hot(A_n)}(A,-)$ computes $\Hm{D(A)}(A,-)$. Therefore, if $M$ is such that $\Hm{D(A)}(A[n], M)$ for all $n$, it follows that $\Hm{\hot(A)}(A[n], M)\cong H^{-n}M=0$ and $M$ is acyclic. Finally, $A$ is compact since taking cohomology commutes with coproducts.
\end{proof}

Any morphism $f\colon A\to B$ of differential graded algebras induces the restriction of scalars dg functor \[\Res_f\colon B\Mod\to A\Mod\]which admits a left adjoint $\Ind_f$ computed as $\Ind_f M=B\otimes_A M$. The functor $\Res_f$ preserves acyclic modules so defines a functor $D(B)\to D(A)$, denoted again with $\Res_f$. The functor $\Ind_f$ does not preserve acyclic modules in general but admits a left derived functor $\Ind_f^{\operatorname{L}} \colon D(A)\to D(B)$ defined by precomposing $\Ind_f$ with $\mathbf{p}$, which is a left adjoint to $\Res_f$.
The morphism $f$ is said to be a derived equivalence if the adjoint pair 
\[\begin{tikzcd}
	{D(A)} & {D(B)}
	\arrow["{\Res_f}"', shift right, from=1-2, to=1-1]
	\arrow["{\Ind_f^{\operatorname{L}}}"', shift right, from=1-1, to=1-2]
\end{tikzcd}\] is an equivalence.
\subsection{Curved algebras}
Let $R$ be a commutative ring.
\begin{definition}[\cite{Positselski_2018}]
A \emph{cdg algebra} $\A$ over $R$ is given by a graded $R$-algebra $\A^\#$ equipped with a derivation $d_\A\in \Hom_R(\A^\#,\A^\#)$ of degree $1$ and an element $c\in A^2$ such that $d_\A(c)=0$ and $d^2_\A=[c, -]$.
\end{definition}
The derivation $d_\A$ is called predifferential and the element $c$ curvature of the algebra. Any dg algebra is in a natural way a cdg algebra by letting $c=0$.
\begin{definition}
    A left \emph{cdg module} over a cdg algebra $\A$ is a graded left $\A^\#$-module $M^\#$ equipped with an $\A^\#$-derivation $d_M\in \Hom_R(M^\#,M^\#)^1$ such that $d^2_M m= cm$ for all $m\in M$. A right cdg $\A$-module is a right graded $\A^\#$-module $N$ with a derivation $d_N\in \Hom_R(N,N)^1$ such that $d^2_N n=- nc$ for all $n\in N$.
\end{definition}
Crucially, as soon as the curvature is not zero, the algebra is not a module over itself; the nonexistence of free modules is one of the main features of our categories of curved modules.

\begin{definition}
    If $M$ and $N$ are cdg left $\A$-modules, the complex $\Hm{\A}(M, N)$ of $\A$-linear morphisms  is defined in the same way as the complex of morphisms between dg modules: it has in degree $n$ the $R$-module of $\A^\# $-module 
   morphisms of degree $n$, and for $f\in \Hom_\A(M,N)^n$ differential defined as $(df)m=d_N (fm)- (-1)^n f(d_M m)$. Even if $M$ and $N$ are not themselves complexes, the space of morphisms is one (i.e. satisfies $d^2=0)$. The dg category $\A\Mod$ has as objects the cdg left $\A$-modules and the just defined complex of morphisms as hom-complex; composition is defined in the obvious way. The same definitions work for right modules, giving rise to the dg category $\RMod\A$ of right cdg $\A$-modules.
\end{definition}
To simplify the notation, we write $\Hm{\A}(M,N)$ instead of $\Hm{\A\Mod}(X,Y)$ for the complex of morphisms inside the dg category $\A\Mod$.
From now on, all modules, unless otherwise specified, will be assumed to be left modules.
\begin{example}[Matrix factorizations]
Let $A$ be any $R$-algebra and $f$ an element in its center. Consider the graded algebra $A[u, u^{-1}]$ with $u$ in degree $2$; this can be made into a cdg algebra $\B$ with zero differential and curvature $c_\B=fu$. Then by definition, a cdg $\B$-module $M$ is given by a 2-periodic graded object \[
\ldots \to M_0 \tow{d_0} M_1 \tow{d_1} M_0 \tow{d_0} M_1 \to \ldots
\] such that $d_id_j=f\id_{M_i}$. These are the objects that are usually called \textit{matrix factorizations} for $f$, with the caveat that the modules appearing in matrix factorizations are usually assumed to be projective. For more about this construction, see e.g. \cite{compgenMF} or \cite{Efimov_2015} for the geometric case.
\end{example}
The category $\A\Mod$ is a pretriangulated dg category; we will denote by $\hot(\A)$ the triangulated category $H^0\A\Mod$. Like in the classical case, triangles in $\hot(\A)$ are graded split short exact sequences. 
\begin{definition}
    Let $\A$ and $\B$ be curved differential graded algebras. A morphism $f\colon \A\to \B$ of cdg algebras is a map \[
    f\colon \A^\# \to \B^\#
    \] of graded algebras which commutes with the predifferentials and carries the curvature of $\A$ to the curvature of $\B$. Just like in the dg case, one has the usual restriction-extension adjoint pair of dg functors between $\A\Mod$ and $\B\Mod$. The restriction of scalars $\Res_f$ is defined in the obvious way and the extension $\Ind_f$ is computed again as $\Ind_f M=\B\otimes_{\A}M$.
\end{definition}
\subsection{Deformations}
Fix a commutative ring $k$. Let $A$ be a dg algebra over $k$. denote by $R_n$ the commutative $k$-algebra $\kn$.
\begin{definition}
    A \emph{cdg deformation} $A_n$ of $A$ of order $n$ is the datum of a structure of a cdg algebra on the $R_n$-module $A\otimes_k R_n\cong A[t]/(t^{n+1})$ which reduces modulo $t$ to the dg algebra structure of $A$.
\end{definition}
Alternatively, as usual one may define a cdg deformation of $A$ of order $n$ to be an $R_n$-free cdg algebra $A_n$ equipped with a dg algebra isomorphism $A_n\otimes_{R_n}k\cong A$. For our purposes the slightly more restrictive definition we gave suffices, but everything we show can be readily carried over to the more general setting.

\begin{remark}
 Usually in deformation theory one considers deformations over arbitrary local artinian algebras, of which the algebras $R_n$ are one example. However the general case makes the theory developed in this paper become significantly more complicated, and we will limit ourselves to the already interesting case of $R_n$.
\end{remark}
     
We will denote by $c$ the curvature of $A_n$ and with $d_{A_n}$ its predifferential. The multiplication of $A_n$ will be denoted with the juxtaposition.

 \subsubsection*{Adjoint functors}For all $0\leq m\leq n$, denote by $A_m$ the $R_m$-algebra \[A_n \otimes_{R_n}R_m\cong \frac{A_n}{t^{m+1}A_n},\] where the $R_n$-action on $R_m$ is defined via the natural surjection $R_n\to R_m$; we will say that $A_m$ is the deformation of order $m$ of $A$ induced by $A_n$. 

The surjection $R_n\to R_m$ induces a surjection
  $A_n\to A_m$,
defining via restriction of scalars the fully faithful forgetful functor \[
F\colon A_m\Mod\to A_n\Mod.
\] This has a left adjoint \[\begin{split}
\Coker t^{m+1}\colon A_n\Mod&\to A_m\Mod\\
M &\to \frac{M}{t^{m+1}M}\cong A_m \otimes_{A_n} M
\end{split}\]
and a right adjoint
\[\begin{split}
\Ker t^{m+1} \colon A_n\Mod&\to A_m\Mod\\
M &\to \Ker t^{m+1}_M\cong \Hm{A_n}(A_m,M),
\end{split}\]
where we have denoted with $t_M$ the action of $t$ on $M$.

\section{The $n$-derived category}\label{parn}
In this section we will introduce the main object studied in this paper (Definition \ref{deffiltered}). From this point onward, let $A$ be a fixed dg algebra over $k$ and $A_n$ a cdg deformation. Given a cdg $A_n$-module $M$, we can define both the $t$-adic filtration \[
0=t^{n+1}M \subseteq t^n M \subseteq \ldots \subseteq tM \subseteq M  
\] and the $K$-filtration \[
0 \subseteq \Ker t_M \subseteq \ldots \subseteq \Ker t^n_M \subseteq \Ker t^{n+1}_M=M. \
\]  Since $c\in tM$, we have that $d^2_M (t^iM)\subseteq t^{i+1}M$ so the successive quotients with respect to the $t$-adic filtration  have induced predifferentials squaring to zero i.e. are actual complexes; since $t\Ker t^i_M\subseteq \Ker t^{i-1}_M$, the same is true for the $K$-filtration. We will write $\Gr_t(M)$ for the associated graded with respect to the $t$-adic filtration, and $\Gr_K(M)$ for that with respect to the $K$-filtration; by the above discussion $\Gr_t(M)$ and $\Gr_K(M)$ are complexes even when $M$ is not. Moreover since $t^i_M$ is a morphism of $A_n$-modules, $t^iM$ and $\Ker t^iM$ are themselves $A_n$-modules. Then, since $t_M$ annihilates the graded pieces of either filtration, the $A_n$-module structure on those reduces to a natural $A$-module structure defining thus functors \[
\Gr_K(-), \Gr_t(-)\colon A_n\Mod \to A\Mod;
\]

Our notion of acyclicity for curved modules will involve the acyclicity of these objects. Explicitly, $\Gr_t(M)$ is acyclic if the $n+1$ complexes \[\Gr_t^i(M)=\frac{t^iM}{t^{i+1}M} \quad i\in {0, \ldots, n}\] are acyclic; analogously, $\Gr_K(M)$ is acyclic if the $n+1$ complexes \[\Gr_K^{i}(M)=\frac{\Ker t^{i+1}_M}{\Ker t^{i}_M}\cong t^{i}\Ker t^{i+1}_M \quad i\in {0, \ldots, n}\] are acyclic; the last isomorphism follows from the obvious fact that, since $\Ker t^n_M\subseteq \Ker t^{n+1}_M$, the kernel of the restricted action of $t^i_M$ on $\Ker t^{i+1}_M$ is still $\Ker t^i_M$.

\begin{lemma}\label{submodules}
        For any cdg $A_n$-module $M$, \[t^iM \cap \Ker t^j_M= t^i\Ker t^{i+j}_M \subseteq M.\]
\end{lemma}
\begin{proof} This is a straightforward elementwise verification.

\end{proof}
The following is a key fact:
\begin{proposition}\label{kacyatcy}
    For any $A_n$-module $M$, the complex $\Gr_t(M)$ is acyclic if and only if the complex $\Gr_K(M)$ is acyclic.
\end{proposition}
\begin{proof}
    Suppose first that $\Gr_K(M)$ is acyclic. The $n$-th quotient of the $K$-filtration is \[\frac{\Ker t^{n+1}_M}{\Ker t^n_M}= \frac{M}{\Ker t^n_M}\cong t^nM\] which is thus acyclic. We then have  the short exact sequence \[
        0 \to \frac{\Ker t_M\cap t^{i-1}M}{\Ker t_M\cap t^iM}\to \frac{t^{i-1}M}{t^{i}M}\tow{t} \frac{t^iM}{t^{i+1}M}\to 0
    \] which by Lemma \ref{submodules} we can rewrite as \begin{equation}\label{fundamentalses}
    0 \to \frac{t^{i-1}\Ker t^i_M}{t^{i}\Ker t^{i+1}_M} \to \frac{t^{i-1}M}{t^{i}M}\tow{t} \frac{t^iM}{t^{i+1}M}\to 0.
    \end{equation} Since $\Gr_K(M)$ is acyclic the first term is the quotient of two acyclic complexes, and by induction we get that each graded piece of the $T$-filtration is acyclic. For the opposite implication the argument is similar: assume that $\Gr_t(M)$ is acyclic. From the short exact sequence \[
    0 \to \frac{\Ker t_M}{t^n M}\to \frac{t^{n-1}M}{t^nM}\tow{t} t^nM \to 0   \] we know that $\Ker t_M$ is acyclic, so again using \eqref{fundamentalses} we conclude by induction that $t^i\Ker t^{i+1}_M$ is acyclic for all $i$.
\end{proof}
\begin{definition}
    An $A_n$-module is said to be \emph{$n$-acyclic} if either of the two equivalent conditions of Proposition \ref{kacyatcy} is satisfied. A closed morphism $f\colon M\to N$ is said to be a \emph{$n$-quasi-isomorphism} if its cone is $n$-acyclic.
\end{definition}
Clearly, a closed morphism is a $n$-quasi-isomorphism if and only if it induces quasi-isomorphisms between the associated graded of either the $t$-adic or the $K$-filtration.
\begin{example}
    If $A_n$ is a dg algebra (i.e. $c=0$), any $n$-acyclic module has by definition a finite filtration with acyclic quotients and is therefore itself acyclic. On the other hand, not all acyclic modules are $n$-acyclic. For an explicit example, consider the first order deformation $R_1=k[t]/(t^2)$ of the base ring, and look at the $R_1$-modules \[
    M=\ldots \to R_1 \tow{t} R_1 \tow{t} R_1 \to \ldots
    \] for an example with $R_1$-free components and its truncation \[
    N= 0 \to k \tow{t} R_1 \to k \to 0 
    \] for a bounded example. As discussed in \cite{Positselski_2018}, the only examples of this kind of behavior for the algebra $R_1$ with $R_1$-free components look like $M$ i.e. rely on the nonregularity of the base ring. As soon as non-free components are let back into the pictures, the class of examples expands dramatically.
\end{example}
\begin{corollary}\label{nacyacy}
    If $M$ is $n$-acyclic, then $\Ker t^i_M$ is $(i-1)$-acyclic for all $i>0$.
\end{corollary}
\begin{proof}
    The graded pieces of $\Ker t^i_M$ with respect to the $K$-filtration are a subset of the graded pieces of $M$ with respect to the same filtration, so this is implied by Proposition \ref{kacyatcy}.
\end{proof}
\begin{remark}
    At a first glance, our constructions have a lot in common with the theory of $N$-complexes \cite{kapranov1996qanalog}; indeed if $A_n$ is a deformation of order $n$, since $d^{2n+2}_M(M)\subseteq t^{n+1}M=0$ any $A_n$-module is a $(2N+2)$-complex. However the two notions of acyclicity differ significantly. Consider the case of a first order deformation $A_1$ with zero curvature, and take an $A_1$-module $M$; looking at it as a $4$-complex, being $4$-exact in the sense of \cite{kapranov1996qanalog} implies that the module \[
    _2H^i(M)=\frac{\Ker d^2_M\colon M^i \to M^{i+1}}{\Img d^2_M \colon M^{i-2}\to M^i}
    \] vanishes. But since $A_1$ has zero curvature, $d_M^2=0$ and $_2H^i(M)=M^i$; therefore any $4$-exact $A_1$ module has to be the zero module. On the other hand any contractible module suffices in giving an example of a nonzero $A_1$ module which is $1$-acyclic according to our definition.
\end{remark}
\begin{proposition}\label{prodcoprod}
    The full subcategory $\operatorname{Ac}^n\subseteq \hot(A_n)$ given by the $n$-acyclic modules is a triangulated subcategory closed under small products and coproducts.
\end{proposition}
\begin{proof}
    We fist verify that $\operatorname{Ac}^n$ is closed under isomorphisms in $\hot(A_n)$, i.e. homotopy equivalences: this follows from the fact that if $f\colon M \to N$ is a homotopy equivalence, it induces homotopy equivalences $\Gr^i_t(f)\colon \Gr^i_t(M)\to \Gr^i_t(N)$. 
    Similarly, if $M$ and $N$ are $n$-acyclic and $f\colon M \to N$ is any closed morphism, we have isomorphisms \[
    \Gr^i_t(\Cn(M \tow{f} N))\cong \Cn( \Gr^i_t(M) \tow{\Gr^i_t(f)} \Gr^i_t(N))
    \] and $\Cn(f)$ is $n$-acyclic.

    We have then proved that $\operatorname{Ac}^n\subseteq \hot(A_n)$ is a triangulated subcategory. To show that it is closed under products and coproducts it is enough to show, for example, that $\Gr_t^i(-)$ commutes with products and coproducts. For this, observe that $\Gr_t^i(-)$ coincides with the composition \[
    A_n\Mod \tow{\img t^i} A_{n-i}\Mod \tow{A\otimes_{A_{n-i}}-}A\Mod;
    \] 
    the functor $\Img t^i$ commutes with both products and coproducts; the tensor functor is a right adjoint so it automatically commutes with coproducts and, since $A$ is finitely-presented as an $A_{n-i}$-module, also commutes with products \cite[\href{https://stacks.math.columbia.edu/tag/059K}{Tag 059K}]{stacks-project}\footnote{The reference is for commutative rings, but in our case it is sufficient since $A\otimes_{A_{n-i}} M\cong k\otimes_{R_{n-i}} M$ as $k$-modules.} .
\end{proof}
\begin{remark}
    In fact, it is immediate to see that $\Gr_K^i(-)$ also commutes with (co)products, so both reductions preserve products and coproducts.
\end{remark}
We now arrive at our main definition.
\begin{definition}\label{deffiltered}
    The \emph{$n$-derived category $D^n(A_n)$} is the quotient of $\hot(A_n)$ by the subcategory $\operatorname{Ac}^n\subseteq \hot(A_n)$.
\end{definition}
In the case $n=0$, the $0$-acyclic modules coincide with the acyclic ones and $D^0(A)=D(A)$.

By Proposition \ref{prodcoprod}, $D^n(A_n)$ is a triangulated category with small products and coproducts  and the quotient functor \[\hot(A_n) \to D(A_n)\] preserves them.

We now give the relevant version of the notions of homotopy projective and homotopy injective modules.
\begin{definition}
     We will say that an $A_n$-module $P$ is \emph{$n$-homotopy projective} if \[
     \Hm{\hot(A_n)}(P, N)=0\
     \] for any $n$-acyclic module $N$. An $A_n$-module $I$ is \emph{$n$-homotopy injective} if \[
     \Hm{\hot(A_n)}(N,I)=0
     \] for any $n$-acyclic module $N$.
 \end{definition}
 In analogy to the classical case, $n$-homotopy projective modules are those in $\prescript{\perp}{} {\operatorname{Ac}^n}$ and $n$-homotopy injective modules those in $\operatorname{Ac}^{n\perp}$.
 \subsection{First results}

 \begin{proposition}
     The module $A\in A_n\Mod$ is $n$-homotopy projective.
 \end{proposition}
 \begin{proof}
     If $M$ is an $n$-acyclic $A_n$-module, we have \[
     \Hm{A_n}(A, M)\cong \Hm{A}(A, \Ker t_M)\cong \Ker t_M
     \]which is acyclic by Lemma \ref{kacyatcy}.
 \end{proof}

We can now show the first way in which $D^n(A_n)$ differs significantly from other derived categories of curved objects in the literature. Observe preliminarily that the forgetful functor $\hot(A)\to \hot (A_n)$ carries acyclic modules to $n$-acyclic modules, so defines a functor $D(A)\to D^n(A_n)$.
\begin{corollary}\label{nonzero}
    The forgetful functor $D(A)\to D^n(A_n)$ is fully faithful.
\end{corollary}
\begin{proof}
    Since $A$ is $n$-homotopy projective as an $A_n$-module, we have isomorphisms \[\Hm{D(A)}(A,M)\cong \Hm{\hot(A)}(A,M)\cong \Hm{\hot (A_n)}(A,M) \cong \Hm{D^n(A_n)}(A,M).
    \]for all $A$-modules $M$. Since $A$ is a compact generator of $D(A)$, by \cite[Lemma 4.2]{Derivingdgcat} this implies the claim.
\end{proof}
Corollary \ref{nonzero} will have a significant generalization in Corollary \ref{embeddings}.
\begin{corollary}\label{embedding}
    If $A_n$ is a dg algebra, there is a fully faithful functor \[
    D(A_n) \hookrightarrow D^n(A_n).
    \]
\end{corollary}
\begin{proof}
Recall that the quotient $\hot(A_n)\to D(A_n)$ admits a fully faithful left adjoint $\mathbf{p
}\colon D(A_n) \to \hot (A_n)$ whose essential image is given by the homotopy projective $A_n$-modules; composing this with the quotient $\hot(A_n) \to D(A_n)$ one obtains a functor \[
D(A_n) \tow{\mathbf{p}} \hot(A_n)\to D^n(A_n).
\] Since every $n$-acyclic module is acyclic, homotopy projective modules are in particular $n$-homotopy projective, so the quotient $\hot(A_n) \to D^n(A_n)$ is fully faithful when restricted to the image of $\mathbf{p}$ and we are done.
\end{proof}
\begin{remark}
    In other words, the quotient functor \[
    D^n(A_n)\to D(A_n)
    \] admits both a left and a right adjoint, which are then automatically fully faithful.
\end{remark}

\section{Compact generation}\label{parcompgen}
So far, we haven't proven anything about $D^n(A_n)$ which is not also true about $\hot(A_n)$; in particular, $D^n(A_n)$ might be too large for it to represent a useful invariant. In this section we prove that that is not the case by showing that, unlike $\hot(A_n)$, the $n$-derived category is a compactly generated triangulated category (Theorem \ref{propcompgen}).

\begin{example}\label{uncurved}
    If $A_n$ is a dg algebra, it is easy to see that $D^n(A_n)$ is compactly generated. Indeed consider the $n+1$ dg $A_n$-modules $A_0, A_1, \ldots,  A_n$. Those are all $n$-homotopy projective, since for any $n$-acyclic  module $M$ one has\[
    \Hm{A_n}(A_i, M)\cong \Hm{A_i}(A_i, \Ker t^{i+1}_M)\cong \Ker t^{i+1}_M
    \] which is $i$-acyclic and hence acyclic.

    Now if $M$ is an $A_n$-module such that $\Ker t^{i+1}_M$ is acyclic for all $i$, it follows immediately that $\Gr_K(M)$ is acyclic and therefore $M$ is $n$-acyclic; compactness comes from the fact that the functors $\Ker t^i$ commute with coproducts.
\end{example}
The problem is that if $A_n$ has nonzero curvature, none of the modules $A_i$ - except $A_0$ - are objects of $A_n\Mod$. The question is then one of finding alternative compact generators that are defined also in the curved case. 
\subsection{Twisted modules}\begin{definition}
    A left (resp. right) \emph{qdg module} \cite{Polishchuk_2012, FilteredAinf} over a cdg algebra $\A$ is a graded left (resp. right) $\A^\#$-module $M^\#$ equipped with a derivation $d_M\in \Hom_R(M^\#,M^\#)$ of degree $1$.
\end{definition}
It is evident from this definition that a cdg module is a qdg module for which the condition $d^2_M=c$ holds. The key reason why qdg modules are relevant is that, while the free module $\A$ is not a cdg module, it is a qdg module. Although for us this notion is just a technical tool to eventually get back to the world of cdg modules, it is one with intrinsic interest: in the many-objects case, considering the larger category of qdg modules makes it possible to define a Yoneda embedding, which does not exist in the cdg case (\cite{FilteredAinf, Polishchuk_2012}).

Let now $M_i$ be a a finite family of qdg $A_n$-modules and \[
F\colon \bigoplus_i M_i \to \bigoplus_i M_i
\] an $A_n$-linear homogeneous morphism of degree $1$.
The twisted module $M_F$ is defined as the qdg module which has as underlying graded module $\oplus_i M_i$ and as predifferential \[
d_F m_i=d_{M_i}m_i+Fm_i.
\]
It is straightforward to see that $d_F$ is still a derivation, so $M_F$ is in a natural way a qdg module; we will say that $M_F$ is the qdg module $\oplus_i M_i$ twisted by $F$. The condition $d_F^2m_i=cm_i$ of being a cdg module corresponds to the Maurer-Cartan equation for $F$ \[
d_{M_i}^2 m_i + Fd_{M_i} + d_{M_i}F + F^2m_i=cm_i.
\]
\begin{remark}
    In the $A_\infty$-case, infinite sums get involved and delicate convergence issues arise, see \cite{LowenCurvature}; in the dg case that we placed ourselves in, any morphism $F$ gives a well defined twisted module.
\end{remark}
\subsection{Construction of the generators}\label{consgen}
To define the modules that will be the generators of $D^n(A_n)$, begin with the following observations: we know that $A_n$, as well as $A_i$ for $i\leq n$, is a qdg $A_n$-module; moreover, there is a well-defined $A_n$-module map \[
t\colon A_{i-1}\to A_i 
\]  for all $i$, together with the already cited projection $\pi \colon A_i\to A_{i-1}$. Finally since $c\in tA_n$, there exists an element $\frac{c}{t}\in A^2_n$ such that $t\frac{c}{t}=c$.

Set $A_{-1}=0$. Define the twisted module $\Gamma_i$ for $i= 0, \ldots, n$ as the qdg module $X_i=A_i\oplus A_{i-1}[1]$ twisted by the matrix \[\gamma_i=
\begin{bmatrix}
0 & \pi\circ \frac{c}{t} \\
t & 0
\end{bmatrix}.\]
Concretely, one has  \[d_{\Gamma_i} (\alpha_k, a_{k+1})=(d_{A_n} \alpha_k+ta_{k+1}, d_{A_{n-1}[1]}a_{k+1}+\pi\left(\alpha_k \frac{c}{t}\right))\]
for $\alpha_k \in A_i$ and $a_{k+1}\in A_{i-1}$; recall that by definition $d_{A_{i-1}[1]}=-d_{A_{i-1}}$. 
\begin{proposition}
    The qdg $A_n$-module $\Gamma_n$ is a cdg $A_n$-module.
\end{proposition}
\begin{proof}
 For simplicity of notation, write $\gamma$ 
 for $\gamma_n$ and $X$ for $X_n$. We have to verify that, for all $x\in X$, \[d_{\Gamma_n}^2x = d_X^2x + d_X \gamma x  +\gamma d_X + \gamma^2 x = cx.\]
    Let's examine the sum part by part.
    \begin{itemize}
        \item Since $\pi(c)$ coincides with the curvature of $A_{n-1}$, we have \[d_X^2(\alpha_k, a_{k+1})=(c\alpha_k-\alpha_k c, \pi (c)a_{k+1}-a_{k+1}\pi(c)).\]
        \item For the last term, we have \[\gamma^2 (\alpha_k, a_{k+1})=(t\pi\left(\alpha_k \frac{c}{t}\right), \pi\left(ta_{k+1}\frac{c}{t}\right))= \left(\alpha_k c, a_{k+1}\pi(c)\right)\] where the last equality is given by direct computation. 
    \end{itemize}
    Therefore $d_X^2+\gamma^2$ equals the action of the curvature, and we will be done if we show that  \[
    d_X \gamma+  \gamma d_X =0.
    \] Computing, we get \[
    d_X \gamma (\alpha_{k}, a_{k+1})=(d_{A_{n}}ta_{k+1}, -d_{A_{n-1}}\pi \left(\alpha_k\frac{c}{t}\right))
    \] and \[
    \gamma d_X (\alpha_k, a_{k+1})=(-td_{A_{n-1}}a_{k+1}, \pi \left(d_{A_n}\alpha_k \frac{c}{t}\right));
    \]since $t$ commutes with the predifferentials, the first terms cancel out; moreover $\pi$ commutes with the predifferentials so the sum of the second terms reduces to \[
    \pi\left(\alpha_k d_{A_n}\frac{c}{t}\right)=\pi(\alpha_k)\pi\left(d_{A_n}\frac{c}{t}\right)
    \]and in order to conclude we need to prove that $\pi\left(d_{A_n}\frac{c}{t}\right)=0$. Since $c$ is closed with respect to $d_{A_n}$, one has \[0=d_{A_n}c=td_{A_n}\frac{c}{t}\]
so $d_{A_n}\frac{c}{t}\in \Ker t_{A_n}$. But $A_n$ is $R_n$-free, so $\Ker t_{A_n}= t^nA_n=\Ker \pi$ and we are done.
\end{proof}
\begin{remark}
    The module $\Gamma_n$ is similar to the module $A_{(n)}$ defined in \cite{LowenCurvature} as the twist of $A_n\oplus A_n[1]$ by the matrix 
    \[
\begin{bmatrix}
0 & \frac{c}{t} \\
t & 0
\end{bmatrix}.\] However in order for $A_{(n)}$ to be a cdg module, the authors need to impose the existence of a deformation $A_{n+1}$ of a higher order extending the deformation $A_n$; this makes it possible to ``correct'' $\frac{c}{t}$ into an actual $d_{A_n}$-closed element. In our case, as we have shown, this need is removed by quotienting out $t^nA_n[1]$. Of course, the module $\Gamma_n$ does not work for the scope of \cite{LowenCurvature} since it is not $R_n$-free.
\end{remark}
In the same way, one proves that $\Gamma_i$ is a cdg $A_n$-module for all $i$. For a cdg $A_n$-module $M$, denote by $(M)_i$ the complex of $R_n$-modules $\Hm{A_n}(\Gamma_i, M)$[1]; explicitly\footnote{The shift is just to avoid negative indices.}, $(M)_i$ is the module $\Ker t_M^n\oplus \Ker t^{n+1}_M[1]$ twisted by the matrix \begin{equation}\label{twistinghom}
    \begin{bmatrix}
0 & \iota \circ \frac{c}{t} \\
t & 0
\end{bmatrix}\end{equation} where $\iota \colon \Ker t^n_M \to \Ker t^{n+1}_M$ is the inclusion; note that $t\Ker t^{n+1}_M\subseteq \Ker t^n_M$ so everything is well defined. By definition $\Gamma_0=A$ so $(M)_0\cong \Ker t_M[1]$.
\begin{proposition}\label{tria}
    There is a triangle in $D(R_n)$
    \begin{equation}\label{triaeq}
\Ker t_M [1]\to (M)_n \to \frac{\Ker t^{n}_M}{t M} \to \Ker t_M [2].         
    \end{equation}
\end{proposition}
\begin{proof}
    There are natural chain maps \[\Ker t_M[1]\to (M)_n\text{ and }(M)_n\to \frac{\Ker t^{n}_M}{t M}\] induced by the inclusion $\Ker t_M\hookrightarrow M$ and the projection $\Ker t^n_M \to \Ker t^n_M/tM$. These give short exact sequences \begin{equation}\label{Mn1}
    0 \to X \to (M)_n \to \frac{\Ker t^{n}_M}{t M} \to 0
    \end{equation}  and \begin{equation}\label{Mn2}
    0\to \Ker t_M[1]\to (M)_n \to Y \to 0
    \end{equation} where $X$ is the twist of $tM\oplus M[1]$ and $Y$ of $M\oplus M/\Ker t_M [1]$ by the same matrix \eqref{twistinghom} that defines $(M)_n$. One also sees that there are chain maps $\Ker t_M[1] \to X$ and $Y \to \Ker t^{n}_M/t M$ - again induced by the inclusion and projection - giving the short exact sequences 
    \[
    0\to \Ker t_M[1]\to X \to Z \to 0
    \] and \[
    0 \to Z'\to Y \to \frac{\Ker t^{n}_M}{t M} \to 0.
    \]
    It's immediate to see that $Z$ and $Z'$ are isomorphic, both being identified to the module $tM \oplus M/\Ker t_M[1]$ twisted by the matrix with the same coefficients as \eqref{twistinghom}. Under the isomorphism $M/\Ker t_M\cong tM$, one sees that $Z$ corresponds to the module $tM\oplus tM[1]$ twisted by the matrix
    \[\begin{bmatrix}
0 &  c\\
1 & 0
\end{bmatrix}.\] Consider the $t$-adic filtration on $Z$; it is a finite filtration, and since $c$ is divisible by $t$ its graded pieces are the twists of $\frac{t^iM}{t^{i+1}M}\oplus \frac{t^iM}{t^{i+1}M}[1]$ by the matrix \[\begin{bmatrix}
0 &  0\\
1 & 0
\end{bmatrix};\] this, being the cone of an isomorphism, is contractible: then each graded piece is acyclic, and so is $Z$.  Thus we get that $X$ is quasi-isomorphic to $\Ker t_M[1]$ and $Y$ to $\Ker t^{n}_M/t M$. Plugging these identifications in either \eqref{Mn1} or \eqref{Mn2} we are done.
\end{proof}
In the same way, one shows that there are triangles in $D(R_i)$ \begin{equation}\label{triagen}
      \Ker t_M[1] \to (M)_i \to \frac{\Ker t^{i}_M}{t \Ker t^{i+1}_M} \to  \Ker t_M[2]
\end{equation} for all $i \leq n$.
\begin{proposition}
    The modules $\Gamma_0, \ldots , \Gamma_n\in A_n\Mod$ are $n$-homotopy projective compact $A_n$-modules.
\end{proposition}
\begin{proof}
    Let $M$ be an $n$-acyclic $A_n$-module: by Proposition \ref{kacyatcy}, we know that $\Ker t_M[1]\cong (M)_0$ is acyclic. Hence by the triangle (\ref{triagen}) it is enough to prove that $\frac{\Ker t^i_M}{t\Ker t_M^{i+1}}$ is acyclic  to conclude that $(M)_i$ is acyclic. For this, consider the short exact sequence\begin{equation*}
        0 \to \frac{t\Ker t^{i+1}_M}{t\Ker t^{i}_M} \to \frac{\Ker t^i_M}{t\Ker t^{i}_M}\to \frac{\Ker t^i_M}{t\Ker t^{i+1}_M} \to 0
    \end{equation*} 
 stemming from the third isomorphism theorem for modules. Applying Lemma \ref{submodules} to $\Ker t^{i+1}_M$, we have that \[t\Ker t^i_M=t\Ker t^{i+1}_M \cap \Ker t^{i-1}_M\] so the first term can be rewritten as \[
 \frac{t\Ker t^{i+1}_M}{t\Ker t^{i}_M} \cong \frac{t\Ker t^{i+1_M}}{t\Ker t^{i+1}_M \cap \Ker t^{i-1}_M}\cong t^i\Ker t^{i+1}_M
 \] and the whole short exact sequence becomes
 \begin{equation}\label{ses1}
        0 \to t^i\Ker t^{i+1}_M \to \frac{\Ker t^i_M}{t\Ker t^{i}_M}\to \frac{\Ker t^i_M}{t\Ker t^{i+1}_M} \to 0;
    \end{equation} 
    by Corollary \ref{nacyacy} $\Ker t^{i+1}_M$ is $i$-acyclic and $\Ker t^i_M$ is $(i-1)$ acyclic, thus the first and second term are acyclic and hence so is the third: we have then proved that $(M)_i$ is acyclic for all $n$-acyclic $M$ and all $i$, i.e. $\Gamma_0, \ldots, \Gamma_n$ are $n$-homotopy projective. Compactness comes directly from the fact that for any collection $\{M_\lambda\}_{\lambda\in \Lambda}$ there is an isomorphism \[ \bigoplus_\lambda(M_\lambda)_i
   \cong  \left(\bigoplus_\lambda M_\lambda\right)_i
   \]following from the fact that the functor $\Ker t^i$ commutes with coproducts.
\end{proof}

\begin{theorem}\label{propcompgen}
    The category $D^n(A_n)$ is compactly generated by the objects $\Gamma_0, \ldots \Gamma_n$.
\end{theorem}
    
\begin{proof}
    By induction on $n$, the case $n=0$ being the usual statement that $A\in D(A)$ is a generator. Let $M$ be a module such that $\Hm{D^n(A_n)}(\Gamma_i, M[l])=0$ for all $i$ and $l\in \mathbb{Z}$. Since the modules $\Gamma_i$ are $n$-homotopy projective, this implies that $  \Hm{A_n}(\Gamma_i, M)$ is acyclic for all $i$. Since the modules $\Gamma_0,  \ldots, $ $\Gamma_{n-1}$ all lie in the image of the forgetful functor from $A_{n-1}\Mod$ and are $(n-1)$-homotopy projective as $A_{n-1}$-modules, we know that \[
   0=\Hm{\hot(A_n)}(\Gamma_i, M[l])\cong \Hm{\hot(A_{n-1})}(\Gamma_i, \Ker t^n_M[l])\cong \Hm{D^{n-1}(A_{n-1})}(\Gamma_i, \Ker t^n_M[l])
   \] for all $l$ and $i<n$, so $\Ker t^n_M$ is $(n-1)$-acyclic by the inductive hypothesis; in particular, $\Ker t_M$ is acyclic. Then, since $(M)_n$ and $\Ker t_M[1]$ are acyclic, the triangle \eqref{triaeq} guarantees that $\frac{\Ker t^{n}_M}{t M}$ is acyclic and by the short exact sequence (\ref{ses1}) for the case $i=n$ this implies that $t^n M$ is acyclic; finally, the $(n-1)$-acyclicity of $\Ker t^n_M$ together with the acyclicity of $t^nM$ means that $\Gr_K(M)$ is acyclic and therefore $M$ is $n$-acyclic.
  \end{proof}
   \begin{example}\label{gradedfield}
       Assume $k$ to be a field, and let $k[u, u^{-1}]$ be the graded field of \cite{MoritaDef} with $u$ in degree $2$; let $k_u[u, u^{-1}]$ be the deformation over $R_1$ corresponding to the Hochschild cocycle $u\in k[u, u^{-1}]^2$. The category $D^1(k_u[u, u^{-1}])$ has as generators the two $k_u[u, u^{-1}]$-modules \[
       \Gamma_0= \ldots \to 0 \to k \to 0 \to k \to 0 \to k \to \ldots
       \] and \[
       \Gamma_1=
     \ldots\to k \tow{t} R_1 \tow{t\,=\,0}k\tow{t}R_1 \to \ldots 
       \]
    It turns out - and this is a peculiarity of the deformation $k_u[u, u^{-1}]$ - that these generators are actually \emph{orthogonal} to each other, i.e. there is no nonzero morphism in either direction. The endomorphism ring of $\Gamma_0$ is isomorphic to the base $k[u, u^{-1}]$ and the one of $\Gamma_1$ is quasi-isomorphic to the same algebra. In Section \ref{parsemiort} we will explain this kind of behaviour in greater generality.
   \end{example}
   \section{Resolutions}\label{parres}
   Using the fact that $D^n(A_n)$ has a set of $n$-homotopy projective compact generators we can already deduce formally the existence of $n$-homotopy projective resolutions. That said, we will still give a direct construction (Corollary \ref{excell}) which gives some insight into the explicit shape of these resolutions. This explicit construction also adapts to give a proof of the existence of $n$-homotopy injective resolutions (Proposition \ref{injres}). 
\subsection{Projective resolutions}

 Let $X=\{X_i\}\subseteq A_n\Mod$ be a set of $A_n$-modules which are compact as objects of $\hot(A_n)$.

\begin{definition}
    An $A_n$-module $P$ is \emph{$X$-semifree} if there exists a filtration \[
    0=F_0P \subseteq F_1 P \subseteq \ldots \subseteq P
    \] such that:
    
    \begin{itemize}
        \item $P=\cup_i F_i P$;
        \item The inclusions $F_iP \hookrightarrow F_{i+1}P$ split as morphisms of graded modules;
        \item $F_{i+1}P/F_iP$ is isomorphic to a direct sum of copies of shifts of elements of $X$.
    \end{itemize}

\end{definition}
We'll say that an $A_n$-module is \emph{$n$-semifree} if it is $\{\Gamma_0, \ldots \Gamma_n\}$-semifree.
\begin{definition}
    A module $M\in \hot(A_n)$ is said to be \emph{$X$-cell} if it lies in the minimal localizing subcategory of $\hot(A_n)$ containing $X$. 
\end{definition}
We'll say that an $A_n$-module $M$ is \emph{$n$-cell} if it is $\{\Gamma_0, \ldots \Gamma_n\}$-cell. A module $N\in \hot(A)$ is \emph{$A$-cell} if it lies in the minimal localizing subcategory of $\hot(A)$ containing the $A$-module $A$; 
\begin{lemma}\label{Xsemifree}
    Any $X$-semifree module is $X$-cell.
\end{lemma}

\begin{proof}
   Same proof as \cite[\href{https://stacks.math.columbia.edu/tag/09KL}{09KL}]{stacks-project}, \textit{mutatis mutandi}: assume that $P$ is an $X$-semifree $A_n$-module. It is easy to see by induction that each $F_iP$ is $X$-cell, and to see that $P$ is itself $X$-cell one uses the existence of the graded split short exact sequence \[
    0 \to \bigoplus_i F_iP \to \bigoplus_i F_iP \to P \to 0.
    \] defined as in \cite[\href{https://stacks.math.columbia.edu/tag/09KL}{09KL}]{stacks-project}.
        
\end{proof}
\begin{remark}
    If an $A_n$-module admits a filtration as in the definition of $X$-semifree with the exception of the quotients being $X$-cell instead of a coproduct of shifted copies of $X_i$, one can use the same argument to prove that it is itself $X$-cell.
\end{remark}
\begin{proposition}
    All $n$-cell $A_n$-modules are $n$-homotopy projective.
\end{proposition}
\begin{proof}
    The first claim follows from the fact that $\Gamma_0,\ldots, \Gamma_n$ are $n$-homotopy projective and that a coproduct of $n$-homotopy projectives is still $n$-homotopy projective.
\end{proof}
In particular, for the case $n=0$ we recover the classical fact that all $A$-cell modules are homotopy projective.
\begin{remark}
It is well known that the property of being $A$-cell coincides with that of being homotopy projective as an $A$-module. By the end of this section, we will see that this generalizes to the filtered setting.
\end{remark}
\begin{lemma}\label{nprojproj}
    If $M\in \hot(A_n)$ is $n$-cell, then both $\Gr_t(M)$ and $\Gr_K(M)$ are $A$-cell and therefore homotopy projective.
\end{lemma}
\begin{proof}
    It is immediate to see that the subcategory of $\hot(A_n)$ given by the modules with $A$-cell associated graded is a triangulated subcategory containing $\Gamma_0, \ldots \Gamma_n$. To conclude we need to show that it is also closed under coproducts, but this follows from the fact that that $\Gr_t(-)$ and $\Gr_K(-)$ commute with coproducts.
\end{proof}

If $M$ is an $A_n$-module, an $n$-cell (resp.  an $n$-semifree) resolution of $M$ is an $n$-cell (resp. an $n$-semifree) $A_n$-module $P$ equipped with a $n$-quasi-isomorphism $P\to M$.
\begin{lemma}\label{cellmodules}
    If $P\to M$ is an $n$-cell resolution of $M\in A_n\Mod$, then \[
    \Gr_t(P)\to \Gr_t(M)
    \] is a homotopy projective resolution of $\Gr_t(M)$ in $\hot(A)$; the same holds for $\Gr_K(-)$.
\end{lemma}
\begin{proof}
   By construction $\Gr_t(P)\to \Gr_t(M)$ is a quasi-isomorphism; moreover by Lemma \ref{nprojproj}, $\Gr_t(P)$ is a homotopy projective $A$-module. The case of $\Gr_K(-)$ is identical.
\end{proof}

The construction of $n$-cell resolutions uses the following procedure which is a generalization of the classical homotopy projective resolution of a dg module due to Keller \cite{Derivingdgcat}\cite[\href{https://stacks.math.columbia.edu/tag/09KK}{09KK}]{stacks-project}.

\subsubsection*{Construction of projective resolutions}

    \begin{lemma}\label{surjection}
        For every $A_n$-module $M$, there exists an $X$-semifree module $Q_X$ with a morphism $Q_X \to M$ inducing surjections
        \[
    \Hm{A_n}(X_i, Q_X)\to \Hm{A_n}(X_i,M).
    \] and
        \[
    \Ker d_{\Hm{A_n}(X_i, Q_X)} \to \Ker d_{\Hm{A_n}(X_i,M)}.
    \] for all $i$.
    \end{lemma}
    \begin{proof}

    By definition, $(\Ker d_{\Hm{A_n}(X_i,M)})^0=Z^0\Hm{A_n}(X_i,M)$. Choose a set $\{f^p\}$ of generators of $Z^0\Hm{A_n}(X_i,M)$ as an abelian group; each $f^p$ defines tautologically a closed morphism \[f^p\colon X_i \to M\] with the property that $f^p$ lies in the image of  \[
     f^p_*\colon Z^0\Hm{A_n}(X_i, X_i) \tow{}Z^0\Hm{A_n}(X_i,M)
    \] as the pushforward of the identity.
    Defining $Q_{X_i}'=\oplus_p X_i$, the collection $\{f^p\}$ defines a closed morphism
    $F\colon Q_{X_i}'\to M$ such that the precomposition with the inclusion in the $p$-th factor $X_i\to \bigoplus_p X_i$ equals $f^p$; it follows that $F$ induces a surjection
    \[Z^0\Hm{A_n}(X_i, Q_{X_i}) \to Z^0\Hm{A_n}(X_i,M).\]
    To get surjectivity at every degree, repeat the same procedure adding appropriate shifts of $X_i$ to $Q_{X_i}'$.
    For the surjectivity on arbitrary morphisms one proceeds similarly: let $C_{X_i}$ be the cocone of the identity of $X_i$. Let $\{g^q\}$ be a set of (not necessarily closed) generators of $\Hm{A_n}(X_i,M)^0$ as an abelian group. The pairs $(g^q, dg^q)$ define closed morphisms $\gamma^q\colon C_{X_i}\to M$, and there is a natural non-closed morphism $\kappa\colon X_i\to C_{X_i}$ with the property that the composition with the inclusion into the $q$-th factor 
    $X_i \tow{\iota} C_{X_i} \tow{\gamma^q} M$ equals $g^q$; therefore $g^q$ lies in the image of the composition \[
    \Hm{A_n}(X_i,X_i)^0 \tow{\kappa_*}\Hm{A_n}(X_i,C_{X_i})^0\tow{\gamma^j_*}\Hm{A_n}(X_i, M)^0
    \]as the image of the identity of $X_i$. Setting again $Q_{X_i}''=\oplus_q C_{X_i}$, we have an induced map $G\colon Q_{X_i}'' \to M$ which gives rise to a surjection \[
    \Hm{A_n}(X_i, Q_{X_i}'')^0 \to \Hm{A_n}(X_i, M)^0;
    \]as before, repeat the whole procedure with shifts of $C_{X_i}$ to obtain surjectivity at all degrees. To get the desired morphism, just set \[Q_X=\bigoplus_i Q_{X_i}'\oplus Q_{X_i}''\] and consider the induced map to $M$; $Q_{X}$ has a $2$-step filtration as in the definition of $X$-semif ree so it is $X$-cell, and the morphism \[
    \Hm{A_n}(X_i,Q_{X_i})\to \Hm{A_n}(X_i, M)
    \] is both surjective and surjective when restricted to the cycles.
            
    \end{proof}
        \begin{proposition}\label{res1}
    For any $A_n$-module $M$ there exists an $X$-semifree module $P_X$ with a closed morphism \[
    P_X\to M
    \] inducing a quasi-isomorphism \[
    \Hm{A_n}(X_i, P_X)\tow{\sim}\Hm{A_n}(X_i, M).
    \] for all $i$.
    
\end{proposition}
\begin{remark}
    If $A_n$ is a dg algebra and by choosing $X=\{A_n\}$, one obtains the same resolution as in \cite{Derivingdgcat}.
\end{remark}
\begin{proof}
    Apply Lemma \ref{surjection} to obtain a closed morphism $P_0\to M$, and denote by $K_0$ its kernel: since both $\Hm{A_n}(X_i,-)$ and $\Ker d_{\Hm{A_n}(X_i,-)}$ are left exact functors, by construction both the sequences \[
    0\to \Hm{A_n}(X_i,K_0)\to \Hm{A_n}(X_i,P_0) \to \Hm{A_n}(X_i,M) \to 0
    \] and \[
    0\to \Ker d_{\Hm{A_n}(X_i,K_0)}\to \Ker d_{\Hm{A_n}(X_i,P_0)} \to \Ker d_{\Hm{A_n}(X_i,M)} \to 0
    \]
    are exact. Applying again Lemma \ref{surjection} to obtain a morphism $P_1 \to K_0$, we can iterate the procedure to obtain a sequence \[
    \ldots \to P_2 \to P_1 \to P_0 \to M
    \] such that 
    \[
    \ldots \to \Hm{A_n}(X_i,P_2) \to \Hm{A_n}(X_i,P_1)  \to \Hm{A_n}(X_i,P_0) \to \Hm{A_n}(X_i,M)\to 0
    \] and 
   \[
    \ldots \to \Ker d_{\Hm{A_n}(X_i,P_2)} \to \Ker d_{\Hm{A_n}(X_i,P_1)}  \to \Ker d_{\Hm{A_n}(X_i,P_0)} \to \Ker d_{\Hm{A_n}(X_i,M)}\to 0
    \] are exact. Set $P_{X_i}=\Tot^\oplus(\ldots \to P_2 \to P_1 \to  P_0)$; the augmentation $P_0\to M$ induces a closed morphism $P_{X_i}\tow{\pi} M$. Since $X_i$ is compact, the cone of $\pi_*$ is isomorphic to \[
    \Tot^\oplus(\ldots \to \Hm{A_n}(X_i,P_2) \to \Hm{A_n}(X_i,P_1)  \to \Hm{A_n}(X_i,P_0) \to \Hm{A_n}(X_i,M)\to 0)
    \] which by \cite[\href{https://stacks.math.columbia.edu/tag/09IZ}{09IZ}]{stacks-project} is exact so\[
    \pi_*\colon \Hm{A_n}(X_i, P_{X_i}) \to \Hm{A_n}(X_i, M)
    \] is a quasi-isomorphism,

 To conclude that $P_X$ is $X$-semifree, recalling that each $P_i$ has a $2$-step filtration $0 \hookrightarrow S_i \hookrightarrow P_i$ as in Proposition \ref{Xsemifree} we can define a filtration $F_i$ on $P_X$ as \[F_{2i}P_X=\Tot(0 \to P_i \to \ldots \to P_0 \to 0)\] and adding $S_{i+1}$ to obtain $F_{2i+1}P_X$. It is then readily seen that the existence filtration $F_i$ satisfies the conditions in the definition of $X$-semifree module.
     
\end{proof}
\begin{corollary}\label{excell}
    Any $A_n$-module $M$ admits an $n$-semifree, and hence $n$-cell, resolution.
\end{corollary}
\begin{proof}
    Apply Proposition \ref{res1} with $X=\{\Gamma_0, \ldots \Gamma_n \}$ to obtain an $n$-semifree module $P$ with a morphism $P\to M$; by construction  \[
    \Hm{A_n}(\Gamma_i,P)\to \Hm{A_n}(\Gamma_i,M)
    \] is a quasi-isomorphism for all $i$. Since the modules $\Gamma_i$ are generators of $D^n(A_n)$, we get that $P\to M$ is an isomorphism in $D^n(A_n)$.
    \end{proof}
\begin{remark}
    As a consequence of the explicit construction, we also obtain that the resolution can be chosen so that \[
    \Hm{A_n}(\Gamma_i, P) \to \Hm{A_n}(\Gamma_i, M)
    \] is a surjection for all $i$; we will use this fact in Section \ref{parmodel}.
\end{remark}
We can finally apply Proposition \ref{krause} to prove the following:
\begin{proposition}\label{resolutions}
    The quotient functor $\hot(A_n) \to D^n(A_n)$ has a fully faithful left adjoint $\mathbf{p}_n$, whose essential image is given by the $n$-cell modules. Moreover, the classes of $n$-cell modules and $n$-homotopy projective modules coincide. 
\end{proposition}
\begin{proof}
    The first statement is a direct application of Proposition \ref{krause}. For the second one, for any $n$-homotopy projective $P$ take an $n$-cell resolution $P_n\to P$; since $P$ is $n$-homotopy projective this has to be a homotopy equivalence and $P$ is itself $n$-cell.
\end{proof}
\begin{remark}
    In particular, we obtain that any $n$-cell $A_n$-module is homotopy equivalent to an $n$-semifree module.
\end{remark}
Since $\Gr_t(-)$ and $\Gr_t(-)$ send $n$-cell modules to $A$-cell modules, we also get the following
\begin{corollary}
    If $M$ is an $n$-homotopy projective $A_n$-module, $\Gr_K(M)$ and $\Gr_t(M)$ are homotopy projective $A$-modules.
\end{corollary}
The converse does not hold; for a simple example, consider the case $A=k$, $A_n=R_n$; then every $A$-module is homotopy projective, but since $D^n(R_n)\neq \hot(R_n)$ - for example, the first category is compactly generated while the second is not - not every $R_n$-module is $n$-homotopy projective.
\begin{corollary}
    The dg subcategory $\mathcal{SF}^n(A_n)\subseteq A_n\Mod$ given by the $n$-semifree $A_n$-modules is a dg enhancement of $D^n(A_n)$.
\end{corollary}
\begin{proof}
    There is a natural functor $H^0\mathcal{SF}^n(A_n) \to D^n(A_n)$ given by the composition \[
    H^0\mathcal{SF}^n(A_n) \hookrightarrow \hot(A_n) \to D^n(A_n);
    \] the right adjoint $\mathbf{p}_n$ gives an inverse.
\end{proof}

\subsubsection*{Filtered behaviors}The forgetful functor $\hot(A_{n-1})\to \hot(A_n)$ carries $(n-1)$-acyclic modules to $n$-acyclic modules and defines a functor \[
\iota_{n}\colon  D^{n-1}(A_{n-1})\to D^n(A_n).
\]
\begin{corollary}\label{embeddings}
    The functor $\iota_{n}$
    is fully faithful.
\end{corollary}
\begin{proof}
    This follows from the fact that the forgetful functor is fully faithful and carries $(n-1)$-cell modules to $n$-cell modules; alternatively, one could note that any morphism from an $A_{n-1}$-module to an $n$-acyclic $A_n$-module $M$ factors through the $(n-1)$-acyclic $A_{n-1}$-module $\Ker t^n_M$ and invoke \cite[Lemma 4.7.1]{krause2009localization}.
\end{proof}

\begin{corollary}\label{corembeddings}
    For any deformation $A_n$, there is a system of embeddings
    \[
D(A)\overset{\iota_{1}}{\hookrightarrow}\ldots D^{n-1}(A_{n-1})\overset{\iota_{n}}{\hookrightarrow} D^n(A_n).
\]
\end{corollary} 
    This behavior is starkly different compared to the classical (dg) case, where homotopy projective modules are very much not preserved and the forgetful functor $D(A_{n-1})\to D(A_n)$ is not fully faithful; having ``all quotients at the same time'' is a fundamental characteristic of this filtered setting.
    \subsection{Injective resolutions}
    One can dualize the argument to also obtain the existence of $n$-homotopy injective resolutions.
   \begin{proposition}\label{injres}
    The quotient functor $\hot(A_n) \to D^n(A_n)$ admits a fully faithful right adjoint $\mathbf{i}_n$, whose essential image is given by the $n$-homotopy injective modules.
\end{proposition}
The construction is somewhat more involved than the projective case and is contained in Appendix \ref{secinj}.

\section{Semiorthogonal decompositions}\label{parsemiort}
In this section we construct a semiorthogonal decomposition of the $n$-derived category (Theorem \ref{filtrationscat}). Ideally we would like each of the factors to be generated by one of the objects $\Gamma_i$ defined in section \ref{consgen}, but it turns out that these have nontrivial morphisms in both directions and we need to slightly tweak the generators. We define \[
G_n=\operatorname{coCone}(\Gamma_n \to \frac{\Gamma_n}{t^n\Gamma_n}).
\]
\begin{lemma}\label{corepresent}
    For any $A_n$-module $M$, there is a quasi-isomorphism \begin{equation}\label{represent}
    \Hm{A_n}(G_n, M)\cong t^nM.
        \end{equation}In particular, $G_n$ is $n$-homotopy projective.
\end{lemma}
\begin{proof}
    One sees that, as graded modules, \[\Hm{A_n}(\Gamma_n, M)\cong \Ker t^n_M[-1] \oplus M\text{ and }\Hm{A_n}\left(\frac{\Gamma_n}{t^n\Gamma_n}, M\right)\cong \Ker t^n_{M}[-1]\oplus \Ker t^n_M\] with predifferentials induced by the usual twisting matrix (\ref{twistinghom}). The functor $\Hm{A_n}(-,M)$ sends cocones to cones, so there is an isomorphism \[
    \Hm{A_n}(G_n, M)\cong \operatorname{Cone}(\Hm{A_n}\left(\frac{\Gamma_n}{t^n\Gamma_n},M\right)\xhookrightarrow{J}\Hm{A_n}(\Gamma_n,M))
    \] where $J$ is given by the identity in the first component and the inclusion $\Ker t^n_M \hookrightarrow M$ in the second. Since $J$ is injective, there is a natural quasi-isomorphism
$\operatorname{Cone }J\cong \Coker J$; looking at the explicit form of $\Coker J$, one sees that it is isomorphic to $M/\Ker t^n_M\cong t^nM$ not only as a graded module but also as complex, since no component of the matrix \eqref{twistinghom} twists the component of the predifferential going from the second factor to itself.
\end{proof}
Recall from Proposition \ref{embeddings} that we had the embedding \[\iota_{n}\colon D^{n-1}(A_{n-1})\hookrightarrow D^n(A_n)\] induced by the forgetful functor: denote by $\iota_{n} D^{n-1}(A_{n-1})\subseteq D^n(A_n)$ its essential image. 
Its left adjoint $M\to \frac{M}{t^nM}$ defines a functor \[\begin{split}
\Coker t^n \colon D^n(A_n) &\to D^{n-1}(A_{n-1})\\
\end{split}\] which, as a consequence of the fact that the forgetful functor carries $(n-1)$-acyclics to $n$-acyclics, is still a left adjoint to $\iota_{n}$.

\begin{lemma}\label{ncategory}
    The subcategory $\iota_{n}D^{n-1}(A_{n-1})\subseteq D^n(A_n)$ is given by the modules $M$ for which $t^nM$ is acyclic.
\end{lemma}
\begin{proof}
    For any module in the image of the forgetful functor one has $t^nM=0$, and since isomorphisms in $D^n(A_n)$ induce quasi-isomorphisms on $\Img t^n$ it's clear that all the modules $M$ in $\iota_{n} D^{n-1}(A_{n-1})$ have the desired property. If instead $M$ is such that $t^nM$ is acyclic, the map $M\to \frac{M}{t^nM}$ is an isomorphism in the $n$-derived category between $M$ and an element in the image of the forgetful functor - $M$ and $\frac{M}{t^nM}$ have the same graded pieces with respect to the $t$-filtration, with the exception of $M$ having the extra piece $t^nM$.
\end{proof}
\begin{definition}
    Define the triangulated subcategory $\T_n\subseteq D^n(A_n)$ as given by the modules $M\in D^n(A_n)$ for which $\frac{M}{t^nM}$ is $(n-1)$-acyclic.
\end{definition}
It follows directly from the definition that $\T_n$ is a localizing subcategory.
\begin{proposition}\label{decomposition}
    There is a semiorthogonal decomposition \[
D^n(A_n)=\langle \iota_{n}D^{n-1}(A_{n-1}), \T_n \rangle
    \] and $\T_n$ coincides with the minimal localizing subcategory of $D^n(A_n)$ containing the module $G_n$.
\end{proposition}
\begin{proof}
First we show that there are no nonzero morphisms from $\T_n$ to $\iota_n D^{n-1}(A_{n-1})$; let $M\in \T_n$ and $N\in \iota_{n} D^{n-1}(A_{n-1})$ and assume that $N\cong\iota_{n} N'$ for some $N'\in D^{n-1}(A_{n-1})$. We have \[
\Hm{D^n(A_n)}(M, N)\cong\Hm{D^n(A_n)}(M, \iota_{n}N')\cong \Hm{D^{n-1}(A_{n-1})}\left(\frac{M}{t^nM}, N'\right)
\] which is $0$ since $M$ lies in $\T_n$. To verify that this is a semiorthogonal decomposition, we construct for any $M$ a triangle \[
P_G \to M \to N \to P_G[1]
\] where $P_G$ is in $\T_n$ and $N$ is in $D^{n-1}(A_{n-1})$. For this, apply Proposition \ref{res1} to $M$ with $X=G_n$ to obtain a $G_n$-cell module $P_G$ with a closed morphism $P_G\to M$ inducing a quasi-isomorphism $t^n P_G\cong t^nM$; denoting with $N$ its cone we know that $t^n N$ is acyclic, so $N$ lies in $D^{n-1}(A_{n-1})$. Finally,  \[
\frac{G_n}{t^n G_n}\cong\operatorname{coCone}\left(\frac{\Gamma_n}{t^n \Gamma_n} \tow{\id}\frac{\Gamma_n}{t^n \Gamma_n}\right)\
\]is contractible so $G_n$ lies in $\T_n$. Since $\T_n$ is localizing, any $G_n$-cell module lies in $\T_n$; this also proves the second claim.
\end{proof}

\begin{proposition}\label{factors}
    The functor $\Img t^n \colon D^n(A_n)\to D(A)$  induces an equivalence 
\[\begin{tikzcd}
	\T_n & {D^n(A_n)} & {D(A).}
	\arrow[hook, from=1-1, to=1-2]
	\arrow["\Img t^n", from=1-2, to=1-3]
	\arrow["\sim", curve={height=-36pt}, from=1-1, to=1-3]
\end{tikzcd}\]
In particular, the functor $\Img t^n$ identifies $D(A)$ with the quotient $\frac{D^n(A_n)}{\iota_n D^{n-1}(A_{n-1})}$.
\end{proposition}
\begin{proof}
    The functor is essentially surjective since it preserves coproducts and it sends $G_n\in \T_n$ to the generator $A\in D(A)$, so we only have to verify full faithfulness; since $G_n$ is a compact generator of $\T_n$ and it is sent to a compact object, it is enough (by the same classical argument as e.g. \cite[Proposition 1.15]{Lunts_2010}) to show that \[
    \Hm{A_n}(G_n, G_n) \to \Hm{A}(A,A)
    \]is a quasi-isomorphism; this will follow once we prove the commutativity of the square
\[\begin{tikzcd}
	{\Hm{A_n}(G_n, G_n)} & {\Hm{A}(A,A)} \\
	{t^nG_n} & A
	\arrow[ from=1-1, to=1-2]
	\arrow["\varphi"', from=1-1, to=2-1]
	\arrow["\sim", from=2-1, to=2-2]
	\arrow["\sim", sloped,  from=1-2, to=2-2]
\end{tikzcd}\]
where the right and lower arrows are the obvious isomorphisms, and $\varphi$ is the quasi-isomorphism of Lemma \ref{corepresent}. To see the explicit form of the morphism $\varphi$, recall that $G_n$ has a graded submodule (not preserved by the predifferential) $A_n\subseteq G_n$ corresponding to the graded submodule $A_n\subseteq \Gamma_n$; denote by $1_n\in G_n$ the element corresponding to the unit of $A_n$. The morphism $\varphi$ is then given explicitly by 
\[\begin{split}
\Hm{A_n}(G_n, M) &\to t^n M\\
[f\colon G_n \to M ] & \to t^nf(1_n);
\end{split}
\]in this form, it is clear that the square commutes and we are done.
\end{proof}
\begin{remark}\label{remrec}
    In fact it is very easy to see that $\iota_{n}D^{n-1}(A_{n-1})\subseteq D^n(A_n)$ is an admissible 
    subcategory: the right adjoint to the embedding is given by the functor $\Ker t^n$ and the left adjoint by $\Coker t^n$; using the description of the quotient $\frac{D^n(A_n)}{\iota_{n}D^{n-1}(A_{n-1})}$ given in Proposition \ref{factors}, we see that there is a recollement
\[\begin{tikzcd}
	{D^{n-1}(A_{n-1})} & {D^{n}(A_{n})} & {D(A).}
	\arrow["\iota_{n}", hook, from=1-1, to=1-2]
	\arrow["{\Img t^n}"{pos=0.6}, from=1-2, to=1-3]
	\arrow[curve={height=-28pt}, from=1-2, to=1-1]
	\arrow[curve={height=28pt}, from=1-2, to=1-1]
	\arrow[curve={height=28pt}, from=1-3, to=1-2]
	\arrow[curve={height=-28pt}, from=1-3, to=1-2]
\end{tikzcd}\]
\end{remark}
\begin{definition}

Define the subcategories $\T_i\subseteq D^n(A_n)$ for $i=0,\ldots, n$ as given by the modules $M$ for which all the graded components $\Gr_t^j(M)$ for $j\neq i$ are acyclic. 
    
\end{definition}
By the same argument as Lemma \ref{ncategory}, the subcategory $\T_0$ coincides with $\iota_{1}D(A)$. 
\begin{theorem}\label{filtrationscat}
    The $n$-derived category $D^n(A_n)$ admits a semiorthogonal decomposition \[
   D^n(A_n)= \langle \T_0, \T_1, \ldots \T_n \rangle
    \]and the functors \[\begin{split}
       \Gr_t^i(-)\colon  D^n(A_n) &\to D(A)\\
    \end{split}\] induce equivalences $\T_i\tow{\sim} D(A)$.
\end{theorem}
\begin{proof}
     This follows by iterating the arguments of Propositions \ref{decomposition} and \ref{factors}; the only further thing to check is that the intersection of $\iota_{n} D^{n-1}(A_{n-1})$ with the subcategory of the modules $M$ for which $\Gr_t^{n-1}(M)$ is $(n-2)$-acyclic is given as claimed by the modules $M$ for which $\Gr_t^i(M)$ is acyclic for $i\neq n-1$, but this follows from reasoning as in Lemma \ref{ncategory}.
\end{proof}
In particular if $A_n$ is a first order deformation of $A$, there is a short exact sequence of categories\footnote{By this we mean that the first arrow is fully faithful and the second is a quotient whose kernel is given by the essential image of the first one.} \[
0 \to D(A) \hookrightarrow D^1(A_1) \tow{M \to tM} D(A) \to 0
\] categorifying 
the usual square zero extension \[
0 \to A \overset{}{\hookrightarrow} A_1 \to A \to 0.
\]corresponding to the deformation $A_1$. 

\subsection{Categorical resolutions}
In this section we will assume that $A_n$ has no curvature, so that we can talk about the classical derived category $D(A_n)$. Recall from \cite{Kuznetsov_2014} that a dg-algebra $A$ is said to be \emph{smooth} if the diagonal bimodule $A\in D(A\otimes \op{A})$ is compact; the same definition holds for a dg-category\footnote{One should be careful about size issues, which we ignore here; for an in depth treatment, see \cite[Appendix A]{Lunts_2010}.}. A \emph{categorical resolution} of a pretriangulated dg-category $\A$ is a smooth pretriangulated dg-category $\C$ equipped with a dg-functor $\A\to \C$ which induces a fully faithful functor between the homotopy categories. One sees that if $A$ is a smooth dg-algebra, then its (dg-enhanced) derived category $D(A)$ is smooth. On the other hand, even when $A$ is smooth, the deformation $A_n$ is never smooth and thus neither is $D(A_n)$; this is due to the fact that the base ring $R_n$ is not reduced and thus is not itself smooth. On the other hand, it turns out that if $A$ is smooth the same holds for $D^n(A_n)$, and we already know from Corollary \ref{embedding} that there is a natural dg-functor $D(A_n)\to D^n(A_n)$ which is fully faithful at the homotopy level. We can therefore prove the following fact:

\begin{proposition}
    If the dg-algebra $A$ is smooth, the $n$-derived category $D^n(A_n)$ is a categorical resolution of $D(A_n)$.
\end{proposition}\begin{proof}
    We prove the case $n=1$, the general case being a similar (but computationally more involved) induction argument. We only have to prove that $D^1(A_1)$ is a smooth dg category. By Proposition \ref{decomposition}, the category $D^1(A_1)$ admits two semiorthogonal compact generators $G_0, G_1$ such that $\Hm{A_1}(G_0, G_0)$ is isomorphic to $A$ and there is a natural quasi-isomorphism \[E_1=\Hm{A_1}(G_1, G_1)\to A.\] By a standard argument, $D^1(A_1)$ is then equivalent to the derived category of the matrix algebra \[
E=\begin{bmatrix}
     A & 0 \\
     X & E_1\\
\end{bmatrix}
\] where $X$ is the $A$-$E_1$ bimodule $\Hm{A_n}(G_0, G_1)$. We know that $A$ is homologically smooth, and that, being quasi-isomorphic to $A$, so is $E_1$. Therefore by \cite[Theorem 3.24]{smootheq} to show that $D^1(A_1)$ is homologically smooth it is sufficient to show that $X$ is perfect as an $A$-$E_1$ bimodule. We can apply \cite[Lemma 2.8]{moduliofob} to show that in order to prove that $X$ is perfect as an $A$-$E_1$ bimodule it is sufficient to see that the induced $A$-module obtained by forgetting the $E_1$-action is perfect.
For that it is enough to look at the explicit form of $X$. Indeed, we have \[
X\cong \Ker t_{G_1} \cong \operatorname{coCone}(\Ker t_{\Gamma_1} \to \Ker t_{\frac{\Gamma_1}{t\Gamma_1}});
\] The module $\Ker t_{\Gamma_1}$ is contractible and $\Ker t_{\frac{\Gamma_1}{t\Gamma_1}}=\frac{\Gamma_n}{t\Gamma_1}$, so there is a natural homotopy equivalence of $A_n$-modules - and thus of $A$-modules \[
X[1] \cong \frac{\Gamma_n}{t\Gamma_1}\cong \operatorname{Cone}\left(\frac{c}{t}\right)
\] where $\frac{c}{t}$ is seen as a closed morphism $A[-1]\to A[1]$; hence there is a triangle in $\hot(A)$ \[
A\to X \to A \to A[1]
\] and $X$ is perfect as an $A$-module.
\end{proof}
\section{Relation with the semiderived category} \label{parsemider}

In this whole section we will assume $k$ to be a field. In \cite{Positselski_2011, Positselski_2018}, Positselski defines several notions of acyclicity for cdg modules; we recall those, and explain the relation to ours (Corollary \ref{corsemider}). 
\begin{definition}{(\cite{Positselski_2011})}
    A cdg $A_n$-module is said to be \emph{absolutely acyclic} if it lies in the minimal thick subcategory of $\hot(A_n)$ containing all totalizations of short exact sequences; it is \emph{coacyclic} if it lies in the minimal triangulated subcategory of $\hot(A_n)$ which contains all totalizations of short exact sequences and is closed under arbitrary coproducts, and \emph{contraacyclic} if it lies in the minimal triangulated subcategory of $\hot(A_n)$ which contains totalizations of short exact sequences and is closed under arbitrary products.

\end{definition}
One key property of contraacyclic modules is that they are right orthogonal to the graded projective $A_n$ modules; dually, coacyclic modules are left orthogonal to the graded injective $A_n$-modules (\cite[Theorem 3.5.1]{Positselski_2011}).

\begin{definition}{(\cite{Positselski_2018})}
    An $R_n$-free $A_n$-module $M$ is said to be \emph{semiacyclic} if $M/tM$ is an acyclic $A$-module. An arbitrary $A_n$-module is \emph{semiacyclic as a comodule} if it lies in the minimal thick subcategory of $\hot(A_n)$ containing all coacyclic modules and all $R_n$-free semiacyclic modules, and \emph{semiacyclic as a contramodule} if it lies in the minimal thick subcategory of $\hot(A_n)$ containing all contraacyclic modules and all $R_n$-free semiacyclic modules. Define the category $\dsi(A_n^{R_n-\operatorname{fr}})$ as the quotient of the homotopy category of $R_n$-free modules by the semiacyclic $A_n$-modules, the category $\dsi(A_n^{\operatorname{co}})$ as the quotient of $\hot(A_n)$ by modules which are semiacyclic as comodules and $\dsi(A_n^{\operatorname{contra}})$ as the quotient of $\hot(A_n)$ by the modules which are semiacyclic as contramodules.

\end{definition}
There are natural functors \[
\dsi(A_n^{R_n-\operatorname{fr}}) \to \dsi(A_n^{\operatorname{contra}}) \text{ and } \dsi(A_n^{R_n-\operatorname{fr}}) \to \dsi(A_n^{\operatorname{co}})
\]
induced by the inclusion of the homotopy category of $R_n$-free $A_n$-modules into $\hot(A_n)$ which are shown in \cite[Theorem 4.2.1]{Positselski_2018} to be equivalences; one refers to either version interchangeably as the \emph{semiderived category} of $A_n$, which is then denoted with $\dsi(A_n)$. In particular, it follows from \cite[Theorem 4.2.1]{Positselski_2018} that an $R_n$-free $A_n$-module is semiacyclic as a co/contramodule if and only if it is semiacyclic as an $R_n$-free module.
\begin{remark}
    In the definitions of \cite{Positselski_2018}, the different versions of the semiderived categories are defined starting from categories of comodules and contramodules; however, since the rings $R_n$ are artinian every module is both a comodule and a contramodule, and the only difference lies in the class of acyclics.
\end{remark}
There is a fundamental difference between $D^n(A_n)$ and the categories that we just presented: since in the semiderived category - and in general, in all the derived categories considered in \cite{Positselski_2011} and \cite{Positselski_2018} - the absolutely acyclic modules are quotiented out, short exact sequences give rise to triangles in the quotient. This crucially is not the case in $D^n(A_n)$, where it is easy to give examples of absolutely acyclic modules which are not $n$-acyclic; consider again the simple case $A=k$, $A_1=R_1$. Then the $R_1$-module \[
M=0 \to k \to R_1 \to k \to 0
\]is absolutely acyclic but not $1$-acyclic. There are nonetheless relations between the two notions, as we now show.
\begin{lemma}\label{Rnfree}
    An $R_n$-free module $M$ is semiacyclic if and only if it is $n$-acyclic.
\end{lemma}
\begin{proof}
    Obviously any $n$-acyclic $R_n$-free module is semiacyclic. For the converse, assume $M$ to be semiacyclic and consider the short exact sequence \[
    0 \to \frac{\Ker t^i_M}{\Ker t^i_M\cap tM }\to \frac{M}{tM} \tow{t^i} \frac{t^iM}{t^{i+1}M}\to 0;
    \] since $M$ is $R_n$-free, $\Ker t^i_M= t^{n-i+1}M$ and $\Ker t^i_M\cap tM= \Ker t^i_M$ so $\Gr_t^i(M)\cong M/tM$ for all $i$, and $M$ is $n$-acyclic.

\end{proof}

As a consequence, the inclusion from the homotopy category of $R_n$-free $A_n$-modules into $\hot(A_n)$ carries semiacyclic modules to $n$-acyclic modules and defines a functor \[\varphi \colon
\dsi(A_n^{R_n\operatorname{-fr}}) \to D^n(A_n).
\]

\subsubsection*{Free resolutions}
In \cite{Positselski_2018}, the author constructs for any $A_n$-module $M$ a triangle in $\hot(A_n)$ \begin{equation}\label{freeres}
M^{\operatorname{fr}} \to M \to C \to M^{\operatorname{fr}}[1]    
\end{equation} where $C$ is contraacyclic and $M^{\operatorname{fr}}$ is $R_n$-free. Concretely, he builds an exact sequence \[
\ldots \to F_2 \to F_1 \to M \to 0
\] where each $F_i$ is $R_n$-free, and defines \[
M^{\operatorname{fr}}=\Tot^\Pi(\ldots\to F_2 \to F_1  \to 0)
\]so that \[
C=\Tot^\Pi(\ldots \to F_2 \to F_1 \to M \to 0)
\] is contraacyclic by \cite[Lemma 4.2.2]{Positselski_2018}. Note that, since for the ring $R_n$ the classes of free and cofree modules coincide, a product of free modules is still a free module. For the same reason, any $R_n$-module admits an injective map into an $R_n$-free module and the argument above dualizes to obtain a triangle \[B \to M \to M^{\operatorname{cofr}} \to B[1]  \] with \[
M^{\operatorname{cofr}}= \Tot^\oplus(0 \to F_1' \to F_2' \to \ldots)
\] where each $F_i'$ is $R_n$-free such that \[
B=\Tot^\oplus(0 \to M \to F_1' \to F_2' \to \ldots)
\] is coacyclic.

\subsection{Derived functors of the reductions}

Consider the functor \[
\Coker t\colon A_n\Mod \to A\Mod,
\] and denote it with $Q$ for simplicity. 

The categories $A_n\Mod$ and $A\Mod$ are dg categories, but their underlying $1$-categories $Z^0A_n\Mod$ and $Z^0A\Mod$ have a natural abelian structure. Seen as a functor between abelian categories, the functor $Q$ is right exact and as such we are interested in its left derived functors. Although we have not shown the category $Z^0 A_n\Mod$ to have enough projectives, $R_n$-free modules are adapted to $Q$ and there are enough $R_n$-free modules, so the left derived functors exist. The next proposition gives a proof of this fact by explicitly describing the higher derived functors.
\begin{proposition}
    The functor $Q$ admits left derived functors $L^iQ$ defined as \[ L^iQ(M)=\begin{cases}\displaystyle Q(M) \text{ for } i=0, \\ \\
        \displaystyle \frac{\Ker t_M}{t^nM} \; \operatorname{ for } i \operatorname{ odd},\\
         \\
       \displaystyle \frac{\Ker t^n_M}{tM}\; \operatorname{ for } i>0  \operatorname{ and \, even.}
    \end{cases}
    \] 
\end{proposition}
\begin{proof}
    We show that, by setting $L^0Q=Q$ and $L^iQ$ as above, the collection $L^*Q$ gives a universal homological $\delta$-functor. Let \[
    0 \to M \to N \to L \to 0
    \] be a short exact sequence of $A_n$-modules. Consider, for any $A_n$-module $K$, the complex of $A_n$-modules \[
K^t= \ldots\to K  \tow{t} K \tow{t^n} K \tow{t} K \to 0. 
\] By definition, one has $H^0K^t=Q(K)$, $H^{-i} K^t= L^iQ(K)$ for $i>0$. Since the morphisms $M\to N$ and $N \to L$ are morphisms of $A_n$-modules, they commute with the action of $t$ and induce a short exact sequence of complexes \[
0 \to M^t \to N^t \to L^t \to 0.
\] The long exact sequence in homology reads \[\ldots \to
\frac{\Ker t^n_L}{tL} \to \frac{\Ker t_M}{t^nM}\to \frac{\Ker t_N}{t^nN} \to \frac{\Ker t_L}{t^nL} \to \frac{M}{tM}\to  \frac{N}{tN}\to  \frac{L}{tL} \to 0.
\] Since everything is functorial, we have proven that our candidate derived functors give rise to a $\delta$-functor. To show that it is universal, we use the fact that it is coeffaceable: indeed, if an $A_n$-module $M$ is $R_n$-free then $\frac{\Ker t_M}{t^nM}=\frac{\Ker t^n_M}{tM}=0$ and, as shown in the proof of \cite[Theorem 4.2.1]{Positselski_2018}, any cdg $A_n$-module admits a surjection from an $R_n$-free module.
\end{proof}
Since $L^iQ(F)=0$ for all $R_n$-free modules $F$ and $i>0$, to compute $L^i(M)$ it is enough to take an $R_n$-free resolution \[
\ldots\to F_2 \to F_1 \to M \to 0,
\]apply the functor $Q$ term-wise and read the functor $L^iQ(M)$ as the horizontal cohomology of the complex of $A$-modules \[
\ldots \to Q(F_2) \to Q(F_1) \to 0 .
\]
\subsubsection*{Deriving the right adjoint}
denote by $K$ the functor $\Ker t\colon A_n\Mod\to A\Mod$. The functor $K$ is left exact, and the arguments of the previous sections dualize to show that $K$ admits right derived functors.
\begin{proposition}\label{dualderived}
    The functor $K$ admits right derived functors $R^iK$, and one has \[R^iK(M)=\begin{cases} K(M) \operatorname{ for } i=0,\\ \\
         \displaystyle R^iK(M)= \frac{\Ker t^n_M}{tM} \;\operatorname{ for } i \operatorname{ odd},\\
         \\
         \displaystyle \frac{\Ker t_M}{t^nM} \;\operatorname{ for } i>0 \operatorname{ and \, even.} 
    \end{cases}
    \] 
\end{proposition}

Peculiarly, it turns out that the right derived functors of $K$ coincide up to a shift in periodicity with the left derived functors of $Q$. 

 We can now prove the main result of this section.
 \begin{proposition}\label{nacysacy}
     Any $n$-acyclic $A_n$-module is semiacyclic both as a comodule and as a contramodule.
 \end{proposition}
\begin{proof}

Assume that $M$ is an $n$-acyclic $A_n$-module, and consider the triangle \eqref{freeres}. Since $C$ is contraacyclic it is in particular semiacyclic as a contramodule, and in order to prove that $M$ is semiacyclic it will be enough to prove that \[
M^{\operatorname{fr}}=\Tot^\Pi(\ldots \to F_2 \to F_1 \to 0)
\] is semiacyclic. This, since $M^{\operatorname{fr}}$ is $R_n$-free, is equivalent to $Q(M^{\operatorname{fr}})$ being acyclic. Clearly \[
Q(M^{\operatorname{fr}})\cong \Tot^\Pi(\ldots \to Q(F_2) \to Q(F_1) \to 0)
\] and by \cite[Proposition 6.2]{EILENBERG19621} we can compute the cohomology of the total complex by taking first horizontal and then vertical cohomology. The horizontal cohomology computes the derived functors $L^iQ(M)$, so the first page reads
\[\begin{tikzcd}
	\ldots & {L^2Q(M)} & {L^1Q(M)} & {L^0Q(M).}
\end{tikzcd}\]
Since $M$ is $n$-acyclic we know that $L^0Q(M)\cong Q(M)$ is acyclic; similarly, it is straightforward to see that since $M$ is $n$-acyclic $L^iQ(M)$ is also acyclic for $i>0$ so we are done. The argument for semiacyclicity as a comodule is dual.
\end{proof}
\begin{remark}
    If $A_n$ has no curvature one can give a different proof of this fact, at least for semiacyclicity as a contramodule. First of all observe that, since $M$ is $n$-acyclic, it is enough to prove that $Q(C)$ is acyclic to conclude that $Q(M^{\operatorname{fr}})$ is acyclic. Then, consider the dg $A_n$-module
\[
L_n=\Tot^\oplus(0 \to A_n \tow{t} A_n \tow{t^n}A_n \tow{t} A_n \to\ldots ).
\] The module $L_n$ is $A_n$-free, so by \cite[Theorem 3.5.1]{Positselski_2011} since $C$ is contraacyclic we know that $\Hm{A_n}(L_n, C)$ is acyclic. Explicitly, we have 
\[
\Hm{A_n}(L_n, M)\cong \Tot^\Pi(\ldots \to C \tow{t} C \tow{t^n} C \tow{t} C \to 0)
\] and we can again compute its cohomology by taking first horizontal and then vertical cohomology. The first page is
\[ \begin{tikzcd}
	{\displaystyle \frac{\Ker t_C}{t^nC}} & {\displaystyle\frac{\Ker t^n_C}{tC}} & {\displaystyle\frac{\Ker t_C}{t^nC}} & {\displaystyle\frac{C}{tC}}
\end{tikzcd}\]
with the natural vertical differentials. Since $M$ is $n$-acyclic and $M^{\operatorname{fr}}$ is $R_n$-free one has that $\frac{\Ker t_C}{t^nC}$ and $\frac{\Ker t^n_C}{tC}$ are acyclic; the second page then reads 
\[\begin{tikzcd}
	0 & 0 & 0 & {\displaystyle H^\bullet\frac{C}{tC}}
\end{tikzcd}\] so $\frac{C}{tC}$ has to be acyclic. This proof does not generalize well to the curved case due to the difficulty of finding explicit noncontractible $A_n$-free cdg modules.
\end{remark}
We have then proven the following:

\begin{corollary}
    The semiderived category $\dsi(A_n)$ is a quotient of the $n$-derived category $D^n(A_n)$.
\end{corollary}
    Explicitly, the semiderived category $\dsi(A_n^{\operatorname{contra}})$ is the quotient of $D^n(A_n)$ by (the closure under filtered quasi-isomorphisms of) the contraacyclic modules; similarly the semiderived category $\dsi(A_n^{\operatorname{co}})$ is the quotient of $D^n(A_n)$ by the closure of the coacyclic modules.
\begin{corollary}\label{corsemider}
    The quotient functor $D^n(A_n) \to \dsi(A_n^{\operatorname{contra}})$ admits a left adjoint $\mathbf{fr}$, defined by assigning to a module $M$ its free resolution $M^{\operatorname{fr}}$. Dually, the quotient functor $D^n(A_n) \to \dsi(A_n^{\operatorname{co}})$ admits a right adjoint $\mathbf{cf}$, defined by assigning to a module $M$ its cofree resolution $M^{\operatorname{cofr}}$.
\end{corollary}
\begin{proof}
We prove the first statement, the second being dual. The claim will follow once we prove that $\Hm{D^n(A_n)}(F,S)=0$ as soon as $F$  is $R_n$-free and $S$ is semiacyclic as a contramodule, since at that point we will be able to use the triangle \eqref{freeres} to apply Proposition \ref{krause}. Any morphism $F \to S$ in $D^n(A_n)$ can be represented as a roof $F \to M \overset{\sim}{\leftarrow} S$ where both arrows are morphisms in $\hot(A_n)$ and $S\to M$ has $n$-acyclic cone; since $S$ is semiacyclic as a contramodule, by Proposition \ref{nacysacy} so is $M$. We will then be done if we prove that any closed morphism between an $R_n$-free $A_n$-module and a module which is semiacyclic as a contramodule factors through an $n$-acyclic module.
The proof of this fact is essentially contained in \cite{Positselski_2018}; indeed by the discussion in the proof of \cite[Theorem 4.2.1]{Positselski_2018} the class $\F$ of modules $M$ for  which any closed morphism $F\to M$ from an $R_n$-free module factors through an $R_n$-free $n$-acyclic module is closed under cones; it is also shown there that any morphism $F \to C$ where $C$ is contraacyclic factors through an $R_n$-free contraacyclic module, and since any $R_n$-free contraacyclic module is $n$-acyclic, all contraacyclic modules lie in $\F$. Finally, all $R_n$-free $n$-acyclic modules lie in $\F$ - just take the trivial factorization $F \to N \tow{\id} N$ - so the class $\F$ contains all the modules that are semiacyclic as contramodules and we are done.
\end{proof}
Recall the functor $\varphi \colon \dsi(A_n^{R_n\operatorname{-fr}}) \to D^n(A_n)$ defined after Lemma \ref{Rnfree}.
\begin{corollary}\label{corsemi}
       The functor $\varphi$ coincides both with the composition \[
       \dsi(A_n^{R_n\operatorname{-fr}}) \tow{\sim} \dsi(A_n^{\operatorname{contra}}) \tow{\mathbf{fr}} D^n(A_n)
       \] as well as with the composition \[
       \dsi(A_n^{R_n\operatorname{-fr}}) \tow{\sim} \dsi(A_n^{\operatorname{co}}) \tow{\mathbf{cf}} D^n(A_n).
       \]In particular, it is fully faithful and admits both a left and a right adjoint.
\end{corollary}
\begin{proof}
    This follows from the fact that if $M$ is $R_n$-free, then one can take $M^{\operatorname{fr}}=M^{\operatorname{cofr}}=M$.
\end{proof}

In other words, $\dsi(A_n)$ identifies with the $R_n$-free part of $D^n(A_n)$. It's straightforward to see that the functors $L^iQ$ carry $n$-acyclic modules to acyclic modules, and thus define functors $D^n(A_n)\to D(A)$. It turns out that we can fully characterize the semiderived category in terms of the functors $L^iQ$. Let us assume for simplicity that $n=1$.
\begin{proposition}
    The subcategory $\dsi(A_1)\subseteq D^1(A_1)$ coincides with the kernel of the functor $L^1Q\colon D^1(A_1)\to D(A)$.
\end{proposition}
\begin{proof}
    We have already shown that the subcategory $\dsi(A_1)\subseteq D^n(A_1)$ coincides with the modules which, up to filtered quasi-isomorphism, are $R_1$-free. If $M\in A_1\Mod$ is $R_1$-free, it is clear that $L^1Q(M)=0$. Vice versa, assume that $L^1Q(M)$ is acyclic. Then, via the usual $R_1$-free resolution, we have an $R_1$-free module $M^{\operatorname{fr}}$ with a morphism $M^{\operatorname{fr}}\to M$ whose cone is isomorphic to \[
    C=\Tot^\Pi(\ldots \to F_2\to F_1 \to M).
    \] We have that \[
    Q(C)\cong \Tot^\Pi(\ldots \to Q(F_2)\to Q(F_1) \to Q(M);
    \] Since $L^1Q(M)$ is acyclic and $L^1Q(M)=L^iQ(M)$ for all $i>0$, by the same spectral sequence argument as in Proposition \ref{nacysacy}, $Q(C)$ is acyclic. Then, since $L^1Q(M^{\operatorname{fr}})=0$, we get that $L^1Q(C)=\frac{\Ker t_C}{tC}$ is acyclic. Finally, by the short exact sequence \[
    0 \to \frac{\Ker t_C}{tC} \to \frac{C}{tC}\to tC \to 0
    \] we get that $C$ is $1$-acyclic and $M^{\operatorname{fr}}\to M$ is an isomorphism in $D^1(A_1)$.
\end{proof}The case $n>1$ is analogous, with the semiderived category coinciding with the intersection of the kernels of the first derived functors of the reductions $\Coker t^i \colon A_n\Mod \to A_{i-1}\Mod$. We omit the proof. 
\begin{remark}
One might wonder whether the two subcategories $\dsi(A_n)$ and $D^{n-1}(A_{n-1})$ generate $D^n(A_n)$ in any way. This is indeed the case for uncurved deformations, but not in general. As is explained in \cite[Example 5.3.6]{Positselski_2018} the semiderived category of the deformation $k_u[u, u^{-1}]$ from Example \ref{gradedfield} is the zero category; on the other hand, we know from Theorem \ref{filtrationscat} that $D^1(k_u[u, u^{-1}])$ cannot coincide with its subcategory $\iota_{1}D(k[u, u^{-1}])$ since their quotient is nontrivial. In fact, one can prove that if the semiderived category is ``large enough'', i.e. there exists a module $M\in \dsi(A_n^{R_n\operatorname{-fr}})$ such that $Q(M)$ is a compact generator of $D(A)$, then the deformation $A_n$ is appropriately equivalent to an uncurved deformation (see \cite{LowenCurvature} for a similar argument). This perspective will be further developed in future work.
    
\end{remark}

\section{A Model structure}\label{parmodel}
In this section we show the existence of a cofibrantly generated model structure on $Z^0A_n\Mod$ presenting the $n$-derived category, which generalizes the classical projective model structure; just like the construction of the resolutions, this will pass through a relative version of the relevant classical construction.

Let $X=\{X_i\}_{i\in \Lambda}$ be a set of finitely presented cdg $A_n$-modules; since $X_i$ is finitely presented, $\Hm{A_n}(X_i, -)$ commutes with directed colimits. 
\begin{definition}
    A closed morphism $f\colon M \to N$ is an \emph{$X$-equivalence} if \[
    f_* \colon \Hm{A_n}(X_i, M) \to  \Hm{A_n}(X_i, N)
    \] is a quasi-isomorphism for all $i$. 
\end{definition} 
If $M$ is any cdg $A_n$-module, denote by $C_M$ the module $\operatorname{coCone}(\id_M)$, so that there is a natural closed morphism $M \to C_M$.
The same proof as \cite[Theorem 5.7]{relative_model} shows the following
\begin{proposition}\label{modelstructure}
    The category $Z^0A_n\Mod$ admits a cofibrantly generated model structure where the weak equivalences are the $X$-equivalences the fibrations are the morphisms $f\colon M \to N$ for which \[
    f_* \colon \Hm{A_n}(X_i, M) \to  \Hm{A_n}(X_i, N)
    \] is surjective for any $i$ and the cofibrations are the maps with the left lifting property with respect to the acyclic fibrations. The generating cofibrations and generating acyclic cofibrations are given by the sets
    \[I=\{X_i[l]\to C_{X_i}[l]\}_{i\in \Lambda,l\in \mathbb{Z}}
\quad\operatorname{ and }\quad
J=\{0 \to C_{X_i}[l]\}_{i\in \Lambda, l\in \mathbb{Z}}.
\]
\end{proposition}

Applying Proposition \ref{modelstructure} to the set $\{\Gamma_0, \ldots, \Gamma_n\}$ we obtain
\begin{proposition}\label{ourmodel}
    The category $Z^0A_n\Mod$ admits a cofibrantly generated model structure where the weak equivalences are the $n$-quasi-isomorphisms. The fibrations are the morphisms $f\colon M \to N$ such that the induced map $\Ker t^i_M \to \Ker t^i_N$ is surjective for all $i\in 1, \ldots, n+1$. The cofibrant objects are the retracts of $n$-semifree modules, and the cofibrations are the graded split morphism with cofibrant cokernel.
\end{proposition}
\begin{proof}
    To verify the description of the fibrations, the only thing that must be checked is that a morphism $f\colon M \to N$ is a fibration for the model structure given by Proposition \ref{modelstructure} if and only if the induced morphism $\Ker t^i_M\to \Ker t^i_N$ is surjective for all $i$. This follows immediately from the isomorphisms of graded modules \[
    \Hm{A_n}(\Gamma_0, M) \cong \Ker t_M \text{ and }\Hm{A_n}(\Gamma_i, M)\cong \Ker t^{i+1}_M\oplus \Ker t^i_M[-1] \text{ for $i>0$.}
    \]One can show by reasoning exactly as in \cite[Proposition 2.5]{relative_model} that all cofibrations are graded split injections with cofibrant cokernel. For the converse we need the following two remarks: first, observe that any cofibrant module $Q$ is in particular $n$-homotopy projective. Indeed if $N$ is an $n$-acyclic module then the natural morphism $C_N\to N$ is an acyclic fibration, and since $Q$ is cofibrant any morphism $Q\to N$ admits an extension $Q\to C_N$ i.e. a nullhomotopy. Secondly, we want to show that\footnote{This is the appropriate version of the notion of graded projectivity.} given a fibration $M\to N$ any arbitrary (not necessarily closed) morphism $f\colon Q\to N$ with $Q$ cofibrant admits a lift $Q\to M$; this follows from the fact that arbitrary morphisms $Q\to M$ correspond to closed morphisms $Q\to C_M$; since $C_M$ is contractible, the map $C_M \to C_N$ induced by $M\to N$ is an acyclic fibration and since $Q$ is cofibrant the required lift exists. At this point we can directly adapt the argument of \cite[Lemma 2.3]{relative_model} to prove that all graded split injections with cofibrant cokernel are cofibrations. We are left with characterizing the cofibrant objects. The modules $\Gamma_i$ and their shifts are cofibrant because they are the cokernels of the generating cofibrations, thus so is an arbitrary coproduct of them. If $P$ is an $X$-semifree module, by the previous discussion we know that the inclusions $F_iP \hookrightarrow F_{i+1}P$ are cofibrations so their transfinite composition $0 \hookrightarrow P$ is also a cofibration, hence $P$ is cofibrant. Vice versa if $Q$ is cofibrant, the resolution $P_Q \to Q$ given by Corollary \ref{excell} is readily seen to be an acyclic fibration, and we know that $P_Q$ is $n$-semifree. Since $Q$ is cofibrant, the diagram 
\[\begin{tikzcd}
	0 & {P_Q} \\
	Q & Q
	\arrow[from=1-2, to=2-2]
	\arrow[from=1-1, to=2-1]
	\arrow[from=1-1, to=1-2]
	\arrow["{\id_Q}", from=2-1, to=2-2]
\end{tikzcd}\] admits a lift, i.e. $Q$ is a retract of $P_Q$.
\end{proof}
It is well-known that the category $Z^0A\Mod$ admits a compactly generated model structure with weak equivalences given by the quasi-isomorphisms and fibrations given by the surjective morphisms. For that model structure, both the adjunctions
\[\begin{tikzcd}
	A\Mod & {A_n\Mod} & {\text{and}} & A\Mod & {A_n\Mod}
	\arrow["F"', shift right, from=1-1, to=1-2]
	\arrow["{\Ker t}"', shift right, from=1-2, to=1-1]
	\arrow["F", shift left, from=1-4, to=1-5]
	\arrow["{\Coker t}", shift left, from=1-5, to=1-4]
\end{tikzcd}\] are Quillen adjunctions. 
\begin{proposition}
    The model structure on $Z^0A\Mod$ is obtained by right transfer from the one on $Z^0A_n\Mod$ along the adjunction
\[\begin{tikzcd}
	A\Mod & {A_n\Mod.}
	\arrow["F", shift left, from=1-1, to=1-2]
	\arrow["{\Coker t}", shift left, from=1-2, to=1-1]
\end{tikzcd}\]
\end{proposition}
\begin{proof}
    We have to prove that a closed morphism $f\colon M \to N$ in $A\Mod$ is a fibration or a weak equivalence precisely when its image via the forgetful functor is. It is immediate to see that the forgetful functor preserves and reflects weak equivalences, so the only question is about the fibrations; this follows from the fact that if $M$ is in the image of $F$, then $\Ker t^i_M=M$ for all $i>0$.
\end{proof}

\section{Filtered torsion derived category of a formal deformation}\label{partorsion}
In this last section we consider the case of a formal deformation $A_t$ over the formal power series ring $k[[t]]$; it is well known \cite{TARRIO19971,Dwyer2002CompleteMA, HomologyComp, POSITSELSKIMGM, LowenCurvature} that for formal deformations the correct category to consider is not the classical derived category. Indeed the theory developed in the artinian case does not generalize verbatim to the formal case because in general, the formal analogue of Lemma \ref{kacyatcy} does not hold: indeed for any torsionfree $A_t$-module $M$ - take for example any $k[[t]]$-free module - one has $\Ker t^i_M=0$ for all $i$, so no information can be inferred on $\Gr_t(M)$ just by knowing $\Gr_K(M)$. Similarly, for any $t$-divisible $A_t$-module - for example, any $k[[t]]$-cofree\footnote{The relevant notion of cofreeness is that of \cite{Positselski_2018}; the prototypical example of a cofree module is the $k[[t]]$-module $k((t))/k[[t]]$.} module - one has that $\Gr_t(M)=0$ regardless of the acyclicity of $\Gr_K(M)$. Therefore to get a meaningful theory we have to place ourselves in a setting where some (homological) form of Nakayama's lemma holds. In this paper we restrict to studying the simpler torsion case, although we do expect that a parallel construction of a filtered derived category of $k[[t]]$-contramodule $A_t$-modules (a version of the complete derived category) is also possible, together with an appropriate version of the co-contra correspondence.
\begin{definition}
    A $k[[t]]$-module $M$ is \emph{torsion} if for every $m\in M$ there exists an $n$ such that $t^nm=0$. An $A_t$-module is torsion if it is torsion as an $k[[t]]$-module. The dg category $\tor{A_n\Mod}$ is the full dg subcategory of $A_n\Mod$ having as objects the torsion $A_t$-modules; we will denote by $\tor{\hot(A_n)}$ the corresponding triangulated homotopy category; since a coproduct of torsion modules is still a torsion module, $\tor{\hot(A_n)}$ is a triangulated category with arbitrary coproducts.
\end{definition}
If $k[t]$ is the coalgebra whose linear dual is the algebra $k[[t]]$, $k[t]$-comodules coincide with torsion $k[[t]]$-modules. 
For torsion modules a version of Nakayama's lemma holds: if $\Ker t_M$ is the zero module, then so is $M$.
\begin{definition}
    A torsion $A_t$-module is said to be \emph{$t$-acyclic} if $\Gr_K(M)$ is acyclic.
\end{definition}
The following should be understood as a homological version of Nakayama's lemma for comodules/torsion modules.
\begin{proposition}
    If $M$ is a $t$-acyclic torsion $A_t$-module, then $\Gr_t(M)$ is acyclic.
\end{proposition}
\begin{proof}
    We first prove that $\frac{M}{tM}$ is acyclic; let $m\in Z^0\frac{M}{tM}$. Since $M$ is torsion, there exists an $n$ such that $m\in \Ker t^n_M$. Since $M$ is $t$-acyclic, $\Ker t^n_M$ is $n$-acyclic so, in particular, $\frac{\Ker t^n_M}{t\Ker t^n_M}$ is acyclic and there exists an $l\in \Ker t^n_M$ such that $m=d(tl)$; so $\frac{M}{tM}$ is acyclic. To conclude that all the graded pieces are acyclic we can reason by induction using the short exact sequence \[
        0 \to \frac{\Ker t_M}{t^{i}\Ker t^{i+1}_M}\to \frac{t^{i-1}M}{t^{i}M}\tow{t} \frac{t^iM}{t^{i+1}M}\to 0.
    \]
\end{proof}
In particular, this proposition tells us that we could have defined a module $M$ to be $t$-acyclic if both $\Gr_t(M)$ and $\Gr_K(M)$ are acyclic. On the other hand, since any $k[[t]]$-cofree module is both torsion and divisible, there are plenty of examples of torsion modules for which the converse does not hold.
\begin{definition}\label{deftderived}
    The \emph{$t$-derived category} of torsion modules $\tor{D^t(A_t)}$
    is defined as the quotient of the homotopy category $\tor{\hot(A_n)}$ by the $t$-acyclic modules.
\end{definition} 
Since $\Gr_K(-)$ commutes with coproducts, the torsion $t$-derived category is a triangulated category with arbitrary coproducts.
\begin{proposition}\label{compform}
     The modules $\Gamma_0, \Gamma_1, \ldots\in \tor{A_t\Mod}$ are homotopy projective compact generators of the category $\tor{D^t(A_t)}$.
\end{proposition}
\begin{proof}
    This has essentially the same proof as the artinian case, since right orthogonality to the modules $\Gamma_i$ is equivalent to the acyclicity of $\Gr_K(-)$.
\end{proof}
In other words, the category $\tor{D^t(A_t)}$ is the colimit in the category of presentable $\infty$-categories (which in this case is the closure under colimits of the naive colimit) of the system of embeddings \[
D(A) \hookrightarrow D_1(A_1) \hookrightarrow \ldots D^n(A_n) \hookrightarrow \ldots
\]
\begin{remark}
    Like in the artinian case, Proposition \ref{compform} implies that the quotient $\tor{\hot(A_n)}\to \tor{D^t(A_n)}$ has a fully faithful left adjoint.
\end{remark}
\begin{proposition}\label{semiform}
There is a left admissible embedding $\dsi(A_t^{\operatorname{co}})\hookrightarrow \tor{D^t(A_t)}$.    
\end{proposition}
\begin{proof}
Again, this has the same proof as in the artinian case, the only modification being the fact that the right derived functors of the left exact functor \[
\Ker t \colon \tor{A_t\Mod}\to A\Mod
\] are now $R^1K(M)= \frac{M}{tM}$ and $R^iK=0$ for $i>1$. To see this, one uses the existence for every short exact sequence \[
0 \to M \to N \to L \to 0
\] of the short exact sequence \[
0 \to \Ker t_M \to \Ker t_N \to \Ker t_L \to \frac{M}{tM} \to \frac{N}{tN} \to \frac{L}{tL} \to 0
\] and the fact that any torsion $A_t$-module admits an inclusion into a $k[[t]]$-cofree, and thus divisible, $A_t$-module.
\end{proof}
Since the algebra $k[[t]]$ is regular, in the case where $A_t$ is a dg deformation the semiderived category with torsion/comodule coefficients corresponds to the classical derived category of torsion $A_t$-modules which is therefore identified with a left admissible subcategory of $\tor{D^t(A_t)}$.

\appendix
\section{Construction of the injective resolutions}\label{secinj}
 In this appendix we prove the existence of $n$-homotopy injective resolutions of $A_n$-modules (Proposition \ref{corinj}). By definition, an $n$-homotopy injective resolution of an $A_n$-module $M$ is given by an $n$-homotopy injective module $I$ equipped with a $n$-quasi-isomorphism $M\to I$. Just like in the classical case this is a bit more delicate than the projective case, and will require some auxiliary constructions.
    \subsubsection*{Right modules}
    Everything that we have proven up until this point for left $A_n$ modules also holds, with opportune modifications (mostly of signs), for right $A_n$-modules; denote by $\RMod A_n$ the dg category of right cdg $A_n$-modules. We have $n+1$ right $A_n$-modules $D_0 \ldots D_n$ - defined in the same way as the left modules $\Gamma_i$ except for the fact that $\frac{c}{t}$ acts on the left in the twisting matrix - which generate the $n$-derived category of right $A_n$-modules. In this setting, $n$-cell modules are those built out of the modules $D_i$ and like in the case of left modules, any right module $M$ admits an $n$-cell resolution $P\to M$.
    \subsubsection*{Linear duality}
    Denote by $\mathbb{Z}\Mod$ the dg category of complexes of abelian groups, and as usual with $\Hm{\mathbb{Z}}(X,Y)$ the complex of $\mathbb{Z}$-linear morphisms; define the linear duality functor $(-)^\vee$ as $\Hm{\mathbb{Z}}(-, \mathbb{Q}/\mathbb{Z})$. One sees that if $M$ is a left $A_n$ module, $M^\vee$ has a natural structure of a right $A_n$-module and vice versa. Since $\mathbb{Q}/\mathbb{Z}$ is injective as an abelian group, the functor $(-)^\vee$ is exact; using in addition the fact that $M^\vee=0$ implies $M=0$, we get that the linear duality functor reflects exactness. 
    
    Applying the functor $(-)^\vee$ to the exact sequence \[
    0 \to \Ker f \tow{f} M \to N \to \Coker f \to 0
    \]one obtains, for any morphism $f$, natural isomorphisms \begin{equation}\label{kercoker}
        \Ker f^\vee \cong (\Coker f)^\vee \text{ and } (\Ker f)^\vee\cong \Coker f^\vee.
    \end{equation}
   If $M$ is a - either left or right - $A_n$-module, there is a canonical evaluation morphism \[\begin{split}
   \ev_M \colon M &\to M^{\vee\vee}=(M^\vee)^\vee\\
   m &\to [\eta \to \eta(m)]
   \end{split}\] which is easily seen to be an injective closed morphism of $A_n$-modules. Moreover, as a consequence of the isomorphisms $\eqref{kercoker}$, there is a natural isomorphism 
   $(\Coker d_M)^{\vee \vee}\cong \Coker d_{M^{\vee \vee}}$ under which the map $\ev_{\Coker d_M}$ corresponds to the map induced by $\ev_M$ between the cokernels of the differentials. Since $\ev_{\Coker d_M}$ is injective, we get that $\ev_M$ induces an injective map between the cokernels of the differential, and thus also an injective map in cohomology.

\begin{definition}
    Define the left $A_n$-modules $\Gamma_i^*=D_i^\vee$. A left $A_n$-module is said to be \emph{$n$-cocell} if it lies in the minimal triangulated subcategory of $\hot(A_n)$ containing $\Gamma_0^*, \ldots, \Gamma_n^*$ which is closed under products.
\end{definition}

It is straightforward to see that if a right $A_n$-module $P$ is $n$-cell, then $P^\vee$ is $n$-cocell.
\subsubsection*{Auxiliary functors}

Define the functors $F_i\colon \RMod A_n \to \mathbb{Z}\Mod$ and $Q_i\colon A_n\Mod \to \mathbb{Z}\Mod$ as \[F_i(M)=\RMod A_n(D_i, M) \text{, }Q_i(M)=D_i\otimes_{A_n} M;\]Explicitly, $F_i(M)$ is given by the qdg module $\Ker t^{i+1}_M \oplus \Ker t^i_M[-1]$ twisted by the matrix \eqref{twistinghom} while $Q_i(M)$ is the qdg module $\frac{M}{t^{i+1}M}\oplus \frac{M}{t^iM}[1]$ twisted by the matrix \[\begin{bmatrix}
0 & \pi\circ \frac{c}{t} \\
t & 0
\end{bmatrix}\]
where $t\colon \frac{M}{t^{i+1}M} \to \frac{M}{t^iM}$ is induced by the action of $t$, and $\pi\colon \frac{M}{t^{i}M}\to \frac{M}{t^{i+1}M}$ is the natural projection.

\begin{lemma}\label{nacydual}
    A left $A_n$-module $M$ is $n$-acyclic if and only if the right $A_n$-module $M^\vee$ is $n$-acyclic.
\end{lemma}

\begin{proof}
   Using again the isomorphism \eqref{kercoker} one sees that the linear duality functor exchanges the graded pieces of the $t$-adic filtration and the $K$-filtration, giving isomorphisms \begin{equation}\label{switchfilt1}
\left(\frac{t^i M}{t^{i+1}M}\right)^\vee=\left(\Ker \frac{M}{t^{i+1}M}\to \frac{M}{t^iM}\right)^\vee\cong \Coker (\Ker t^i_{M^\vee} \to \Ker t^{i+1}_{M^\vee})=\frac{\Ker t^{i+1}_{M^\vee}}{\Ker t^i_{M^\vee}} 
\end{equation} and \begin{equation}\label{switchfilt2}
\left(\frac{\Ker t^{i+1}_{M}}{\Ker t^i_{M}}\right)^\vee = \left(\Coker (\Ker t^i_{M} \to \Ker t^{i+1}_{M})\right)^\vee\cong \Ker \frac{M^\vee}{t^{i+1}M^\vee} \to \frac{M^\vee}{t^i M^\vee} =\frac{t^i M^\vee}{t^{i+1}M^\vee}.
\end{equation}

If $M$ is $n$-acyclic, by \eqref{switchfilt1} and the fact that $(-)^\vee$ is exact we get that $M^\vee$ is also $n$-acyclic. The other implication follows from the same argument together with fact that $(-)^\vee$ reflects exactness.
\end{proof}
\begin{lemma}\label{functorprop}
    The functors $F_i$ and $Q_i$ have the following properties:
    \begin{enumerate}
    \item There are natural isomorphisms  \begin{equation}\label{funprop}
    F_i(M)^\vee \cong Q_i(M^\vee) \text{ and } Q_i(M)^\vee\cong F_i(M^\vee).
    \end{equation}
    \item Under the isomorphism $Q_i(M)^{\vee \vee}\cong F_i(M^\vee)^\vee\cong Q_i(M^{\vee \vee})$ the map \[
    Q_i(\ev_M)\colon Q_i(M)\to Q_i(M^{\vee \vee})\] corresponds to \[
    \ev_{Q_i(M)}\colon Q_i(M) \to Q_i(M)^{\vee \vee}.
    \]
    In particular, $Q_i(\ev_M)$ is injective and induces an injection between the cokernels of the differentials.
    \item The functors $Q_i$ are right exact and preserve products;
    \item A left $A_n$-module $M$ is $n$-acyclic if and only if $Q_i(M)$ is acyclic for all $i$.
    \end{enumerate}
\end{lemma}
\begin{proof}
    The proofs of points $1)$ and $2)$ consist only in writing down the explicit forms of $F_i$ and $Q_i$ and repeatedly applying the isomorphisms \eqref{kercoker}. The fact that $Q_i$ preserves products follows again from the fact that $A_i$ is finitely presented as an $A_n$-module, so the functor $A_i\otimes_{A_n} -$ preserves products. Since $F_i$ is defined as a hom-functor it is left exact, and the fact that $Q_i$ is right exact follows then from point $1)$ together with the fact that the linear duality functors preserves and reflects exactness: this proves point $3)$. Finally, point $4)$ follows again from point $2)$ together with the facts that the modules $R_i$ are $n$-homotopy projective generators of the $n$-derived category and Lemma \ref{nacydual}.
\end{proof}
\begin{lemma}
    Any $n$-cocell module is $n$-homotopy injective.
\end{lemma}
\begin{proof}
    This will follow once we prove that the modules $\Gamma_i^*=R_i^\vee$ are $n$-homotopy injective. For that, we use that there are isomorphisms \[
    \Hm{A_n\Mod}(M, R_i^\vee)\cong \Hm{\RMod A_n}(R_i, M^\vee)
    \] defined by assigning to a morphism $f\colon M\to  R_i^\vee$ the morphism 
    \[\begin{split}
        R_i &\to M^\vee=\Hm{\mathbb{Z}}(M,\mathbb{Q}/\mathbb{Z})\\
        r &\to [m \to f(m)(r)].
    \end{split}\] 
    The claim then follows from Lemma \ref{nacydual}.
\end{proof}
As before, an $n$-cocell resolution is a closed morphism of $A_n$-modules $M \to I$ where $I$ is $n$-cocell with $n$-acyclic cone; it is clear that an $n$-cocell resolution is in particular an $n$-homotopy injective resolution.
\subsubsection*{Construction of the injective resolutions}
We are now ready to prove the following fact:
\begin{proposition}\label{corinj}
    Any left $A_n$-module $M$ admits an $n$-cocell resolution $M \to I$.
\end{proposition}
\begin{proof}
    Consider the right $A_n$-module $M^\vee$; recalling that $F_i(-)=\RMod A_n(R_i, -)$, apply to $M^\vee$ the version for right modules of Lemma \ref{res1} with $X=\oplus R_i$ to obtain a closed morphism\[
    P \tow{p} M^\vee 
    \] with $P\in n$-cell such that $F_i(P)\tow{F_i(p)} F_i(M)$ and $\Ker d_{F_i(P)} \tow{F_i(p)}\Ker d_{F_i(M)}$ are surjective for all $i$. Dualizing, we get a map \[
    M^{\vee\vee} \tow{p^\vee} P^\vee
    \] with the property that $F_i(M^\vee)^\vee\tow{F_i(p)^\vee} F_i(P)^\vee$ and $(\Ker d_{F_i(M^\vee)})^\vee \tow{F_i(p)^\vee} (\Ker d_{F_i(M^\vee)})^\vee$ are injections. Using Lemma \ref{functorprop} $1)$ and the fact that under the isomorphism \eqref{funprop} we have $F_i(p)^\vee=Q_i(p^\vee)$, we get that $p^\vee$ induces injections $Q_i(M^{\vee \vee}) \tow{Q_i(p^\vee)} Q_i(P^\vee)$ and $\Coker d_{Q_i(M^{\vee\vee)}}\tow{Q_i(p^\vee)} \Coker d_{Q_i(P^\vee)}$. Compose now $p^\vee$ with $\ev_M$ to obtain a map \[
    M \tow{\ev_M} M^{\vee \vee} \tow{p^\vee} P^\vee.
    \] Set $P^\vee=I_0$; Since $\ev_M$ induces injections on $Q_i$ and on the cokernels of the differentials, we have that the induced maps $Q_i(M) \to Q_i(I_0)$ and $\Coker d_{Q_i(M)}\to \Coker d_{Q_i(I_0)}$ are injective.
    Moreover since $P$ is $n$-cell, $I_0=P^\vee$ is $n$-cocell. denote by $C_0$ the cokernel of $M\to I_0$; since $Q_i$ and $\Coker d_{Q_i(-)}$ are right exact functors, the sequences \[
    0 \to Q_i(M) \to Q_i(I_0) \to Q_i(C_0) \to 0
    \] and \[
    0 \to \Coker d_{Q_i(M)} \to \Coker d_{Q_i(I_0)} \to \Coker d_{Q_i(C_0)}\to 0
    \] are exact. Applying again the procedure described above to obtain a morphism $C_0 \to I_1$, we can iterate the construction to get a sequence \[
    M \to I_0 \to I_1 \to I_2 \to \ldots
    \] such that each $I_k$ is $n$-cocell and the sequences  \[
    0\to Q_i(M) \to Q_i(I_0)\to Q_i(I_1) \to Q_i(I_2)\to \ldots
    \] and \[
    0\to \Coker_{Q_i(M)} \to \Coker_{Q_i(I_0)}\to \Coker_{Q_i(I_1)} \to \Coker_{Q_i(I_2)}\to \ldots
    \] are exact. Setting $I=\Tot^\Pi (I_\bullet)$, we have a natural map $M \to I$ whose cone is isomorphic to\[
    T=\Tot^\Pi(
    0\to M \to I_0 \to I_1 \to I_2 \to \ldots
    );
    \]Since $Q_i$ commutes with cones and products, we get that \[Q_i(T)\cong \Tot^\Pi(0\to Q_i(M) \to Q_i(I_0) \to Q_i(I_1) \to Q_i(I_2) \to \ldots)\] which by \cite[\href{https://stacks.math.columbia.edu/tag/09J0}{09J0}]{stacks-project} is acyclic; therefore $T$ is $n$-acyclic, and $M\to I$ is an isomorphism in $D^n(A_n)$. The proof that $I$ is $n$-cocell is completely dual to the projective case, consisting of an application of the dual of Lemma \ref{Xsemifree} (see \cite[\href{https://stacks.math.columbia.edu/tag/09KR}{09KR}]{stacks-project}).
\end{proof}
This concludes the proof of Proposition \ref{injres}.
\begin{corollary}
    The classes of homotopy injective and $n$-cocell modules coincide. In particular if $M$ is an $n$-homotopy injective $A_n$-module, then $\Gr_t(M)$ and $\Gr_K(M)$ are homotopy injective $A$-modules.
\end{corollary}
\begin{remark}
    The fact that our construction for injective resolutions is somewhat more involved than the one for projective resolutions is due to the fact that, while the generators $\Gamma_i$ are compact, the cogenerators $\Gamma_i^*$ are \textit{not} cocompact - cocompact objects rarely exist in module categories. This is why we had to introduce the functors $Q_i$; indeed while the complex $\Hm{A_n}(\prod_k M_k, \Gamma_i^*)\cong  Q_i(\prod_k M_k)^\vee$ seems hard to control, the fault lies only in the linear dual - once we manage to get rid of it, the functor $Q_i$ is as well behaved as one could hope.
\end{remark}
    Just like we could have proven the existence of projective resolutions using Brown representability, it is possible to use a similar strategy for injective resolutions; however, since as we discussed the cogenerators $\Gamma_i^*$ are not cocompact, some care is needed. The correct ingredient to use would seem to be the notion of $0$-compactness  introduced in \cite{oppermann2018change} which is strictly related to the fact that, while it is not possible to write down $\Hm{A_n}(\prod_k M_k, \Gamma_i^*)$ in terms of $\Hm{A_n}( M_k, \Gamma_i^*)$, it is true that if $\bigoplus_k \Hm{A_n}(M_k, \Gamma_i)$ is acyclic the same is true for $\Hm{A_n}(\prod_k M_k, \Gamma_i^*)$ and vice versa. Using the fact that the modules $\Gamma_i^*$ are $0$-cocompact, Theorem 6.6 of \cite{oppermann2018change} yields the existence of the desired resolution.

\printbibliography
\end{document}